\documentclass[12pt]{amsart}%
\usepackage{amsmath}
\usepackage{amsfonts}
\usepackage{amssymb}
\usepackage{amstext}
\usepackage{graphicx}%
\providecommand{\U}[1]{\protect\rule{.1in}{.1in}}
\providecommand{\U}[1]{\protect \rule{.1in}{.1in}}
\providecommand{\U}[1]{\protect \rule{.1in}{.1in}}
\newtheorem{theorem}{Theorem}[section]
\newtheorem{lemma}{Lemma}[section]
\newtheorem{proposition}{Proposition}[section]
\newtheorem{corollary}{Corollary}[section]

\newtheorem{definition}{Definition}[section]

\numberwithin{equation}{section}

\theoremstyle{remark}
\newtheorem{remark}{Remark}[section]
\numberwithin{equation}{section}
\begin{document}
\title[Li-Yau gradient estimate for sum of squares of vector fields]{On Li-Yau gradient estimate for sum of squares of vector fields up to higher step}
\author{Der-Chen Chang}
\address{Department of Mathematics, Georgetown University, Washington DC 20057-0001, U.S.A.}
\email{chang@georgetown.edu }
\author{Shu-Cheng Chang}
\address{Department of Mathematics and Taida Institute for Mathematical Sciences
(TIMS), National Taiwan University, Taipei 10617, Taiwan, R.O.C}
\email{scchang@math.ntu.edu.tw }
\author{Chien Lin}
\address{Department of Mathematics, National Tsing Hua University, Hsinchu 30013,
Taiwan, R.O.C}
\email{r97221009@ntu.edu.tw}
\thanks{The first author is partially supported by an NSF grant DMS-1408839 and Hong
Kong RGC competitive earmarked research grant $\#$601410. The second and the
third author are partially supported by the MOST of Taiwan}
\subjclass{Primary 32V05, 32V20; Secondary 53C56}
\keywords{CR Li-Yau gradient estimate, Sum of squares of vector fields, H\"{o}rmander's
condition, curvature-dimension inequality, CR Harnack inequality, CR heat kernel.}

\begin{abstract}
In this paper, we generalize the Cao-Yau's gradient estimate for the sum of
squares of vector fields up to higher step under assumption of the generalized
curvature-dimension inequality. With its applications, by deriving a
curvature-dimension inequality, we are able to obtain the Li-Yau gradient
estimate for the CR heat equation in a closed pseudohermitian manifold of
nonvanishing torsion tensors. As consequences, we obtain the Harnack
inequality and upper bound estimate for the CR heat kernel.

\end{abstract}
\maketitle

\section{Introduction}

One of the goals for differential geometry and geometric analysis is to
understand and classify the singularity models of a nonlinear geometric
evolution equation, and to connect it to the existence problem of geometric
structures on manifolds. For instance in 1982, R. Hamilton (\cite{h3})
introduced the Ricci flow. Then by studying the singularity models (\cite{h2},
\cite{pe1}, \cite{pe2}, \cite{pe3}) of Ricci flow, R. Hamilton and G. Perelman
solved the Thurston geometrization conjecture and Poincare conjecture for a
closed $3$-manifold in 2002.

On the other hand, in the seminal paper of P. Li and S.-T. Yau (\cite{ly})
established the parabolic Li-Yau gradient estimate and Harnack inequality for
the positive solution of heat equation
\[%
\begin{array}
[c]{c}%
(\Delta-\frac{\partial}{\partial t})u\left(  x,t\right)  =0
\end{array}
\]
in a complete Riemannian manifold with nonnegative Ricci curvature. Here
$\Delta$ is the time-independent Laplacian operator. Later, R. S. Hamilton (
\cite{h1}) obtained the so-called Li-Yau-Hamilton inequality for the Ricci
flow in a complete Riemannian manifold with a bounded and nonnegative
curvature operator. Recently, G. Perelman (\cite{pe1}) derived the remarkable
entropy formula which is important in the study of the singularity models of
Ricci flow. The derivation of the entropy formula resembles the Li-Yau
gradient estimate for the heat equation. Since then, there were many
additional works in this direction which cover various different geometric
evolution equations such as the mean curvature flow ( \cite{h4}), the
K\"{a}hler-Ricci flow (\cite{ca}), the Yamabe flow (\cite{ch} ), etc.

In the paper of \cite{ckw}, following this direction, we propose to study the
most important geometrization problem of closed CR $3$-manifolds via the CR
torsion flow (\ref{0}). More precisely, let us recall that a strictly
pseudoconvex CR structure on a pseudohermitian $3$-manifold $(M,J,\theta)$ is
given by a cooriented plane field $\ker\theta$, where $\theta$ is a contact
form, together with a compatible complex structure $J$. Given this data, there
is a natural connection, the so-called Tanaka-Webster connection or
pseudohermitian connection. We denote the torsion of this connection by
$A_{J,\theta}$, and the Webster curvature by $W$. We consider the torsion
flow
\begin{equation}
\left\{
\begin{array}
[c]{l}%
\frac{\partial J}{\partial t}=2A_{J,\theta},\\
\frac{\partial\theta}{\partial t}=-2W\theta,
\end{array}
\right. \label{0}%
\end{equation}
on $(M,J,\theta)\times\lbrack0,T).$ It is the negative gradient flow of CR
Einstein-Hilbert functional. \ Along this direction with the torsion flow
(\ref{0}), we have established the CR Li-Yau gradient estimate (\cite{ckl})
and the Li-Yau-Hamilton inequality (\cite{cftw}, \cite{ccf}) for the positive
solution of CR heat equation%

\begin{equation}%
\begin{array}
[c]{c}%
(\Delta_{b}-\frac{\partial}{\partial t})u\left(  x,t\right)  =0
\end{array}
\label{heat eq}%
\end{equation}
in a closed pseudohermitian $(2n+1)$-manifold with nonnegative pseudohermitian
Ricci curvature and vanishing torsion tensors (see next section for
definition). Here $\Delta_{b}$ is the time-independent sub-Laplacian operator.
One of our goals in this paper is to find the CR Li-Yau gradient estimate in a
closed pseudohermitian $(2n+1)$-manifold with nonvanishing torsion tensors.

Let us start with a more general setup for the Li-Yau gradient estimate in a
closed manifold with a positive measure and an operator
\begin{equation}
L=\sum_{j=1}^{d}e_{j}^{2}\text{ }\label{2011c}%
\end{equation}
with respect to the sum of squares of vector fields $e_{1},\ e_{2}%
,\ ...,e_{d}$ which satisfies H\"{o}rmander's condition (\cite{h}). More
precisely, the vector fields $e_{1},\ e_{2},\ ...,e_{d}$ together with their
commutators $Y_{1},...,Y_{h}$ up to finite order span the tangent bundle at
every point of $M$ with $d+h=\dim M.$ It is to say that the commutators of
$e_{1},\ e_{2},\ ...,e_{d}$ of order $r$ ( or called step $r$ as well) can be
expressed as linear combinations of $e_{1},\ e_{2},\ ...,e_{d}$ and their
commutators up to the order $r-1.$ The very first paper of H.-D. Cao and S.-T.
Yau (\cite{cy}) follows this line, and considers the heat equation
\begin{equation}
(L-\frac{\partial}{\partial t})u\left(  x,t\right)  =0.\label{2011b}%
\end{equation}
They derived the gradient estimate of sum of squares of vector fields of step
two ($r=2$) in a closed manifold with a positive measure.

In this paper, with the help of a generalized curvature-dimension inequality
explained below, we are able to obtain the Li-Yau gradient estimate for the CR
heat equation in a closed pseudohermitian manifold of the nonvanishing torsion
tensor. As consequences, we obtain the Harnack inequality and upper bound
estimate for the heat kernel. With the same mentality, we generalize the
Cao-Yau's gradient estimate for the sum of squares of vector fields up to
order three and higher under assumption of a generalized curvature-dimension inequality.

One of the key steps in Li-Yau's method for the proof of gradient estimates is
the Bochner formula involving the (Riemannian) Ricci curvature tensor. Bakry
and Emery (\cite{be}) pioneered the approach to generalizing curvature in the
context of gradient estimates by using curvature-dimension inequalities. In
the CR analogue of the Li-Yau gradient estimate (\cite{ckl}), the CR Bochner
formula (\cite{g}) is
\begin{equation}%
\begin{array}
[c]{ccl}%
\frac{1}{2}\Delta_{b}\left\vert \nabla_{b}f\right\vert ^{2} & = & \left\vert
Hess(f)\right\vert ^{2}+\left\langle \nabla_{b}f,\nabla_{b}(\Delta
_{b}f)\right\rangle +2\left\langle J\nabla_{b}f,\nabla_{b}f_{0}\right\rangle
\\
&  & +(2Ric-\left(  n-2\right)  Tor)(\left(  \nabla_{b}f\right)  _{C},\left(
\nabla_{b}f\right)  _{C}),
\end{array}
\label{2015}%
\end{equation}
which involves a term $\left\langle J\nabla_{b}f,\text{ }\nabla_{b}%
f_{0}\right\rangle $ that has no analogue in the Riemannian case. Here
$f_{0}:=\mathbf{T}\varphi$ and $T$ is the characteristic vector field. In
order to deal with the extra term $\left\langle J\nabla_{b}f,\text{ }%
\nabla_{b}f_{0}\right\rangle $ in case of vanishing torsion tensors, based on
the CR Bochner formula (\ref{2015}), we can show the so-called
curvature-dimension inequality (see Lemma \ref{l1}):
\begin{equation}
\Gamma_{2}(f,f)+\nu\Gamma_{2}^{Z}(f,f)\geq\frac{2}{n}\left\vert \Delta
_{b}f\right\vert ^{2}+\left(  -2k-\frac{8}{\nu}\right)  \left\vert \nabla
_{b}f\right\vert ^{2}+2n\left\vert f_{0}\right\vert ^{2}\label{2016B}%
\end{equation}
for any smooth function $f\in C^{\infty}(M)$ and $\nu>0$ and the
pseudohermitian Ricci curvature bounded below by $-k$. Here
\[
\Gamma_{2}^{Z}(f,\ f):=2\left\vert \nabla_{b}f_{0}\right\vert ^{2}%
\]
and
\[
\Gamma_{2}(f,\ f):=4\left\vert Hess(f)\right\vert ^{2}+8Ric(\left(  \nabla
_{b}f\right)  _{C},\left(  \nabla_{b}f\right)  _{C})+8\left\langle J\nabla
_{b}f,\nabla_{b}f_{0}\right\rangle .
\]

Before we introduce the generalized curvature-dimension inequality
(\ref{2016A}) which was first introduced by Baudoin and Garofalo (\cite{bg})
in the content of sub-Riemannian geometry, it is useful to compare Cao-Yau's
notations with pseudohermitian geometry.

Let $J$ be a CR structure compatible with the contact bundle $\xi=\ker\theta$
and $\mathbf{T}$ be the characteristic vector field of the contact form
$\theta$ in a closed pseudohermitian $(2n+1)$-manifold $(M,J,\theta)$ . The CR
structure $J$ decomposes $\mathbf{C}\otimes\xi$ into the direct sum of
$T_{1,0}$ and $T_{0,1}$ which are eigenspaces of $J$ with respect to $i$ and
$-i$, respectively.\ By choosing a frame $\left\{  \mathbf{T},Z_{j}{}%
_{i},Z_{\overline{j}}\right\}  $ of $TM\otimes%
\mathbb{C}
$ \ with respect to the Levi form such that
\[
J(Z_{j})=iZ_{j}\ \text{\ \textrm{and}\ }J(Z_{\overline{j}})=-iZ_{\overline{j}%
},
\]
then $Y_{1}$ will be the characteristic vector field $\mathbf{T}$ with
$\alpha=1$, $d=2n$ and
\[
Z_{j}=\frac{1}{2}(e_{j}-ie_{\widetilde{j}})\text{ \textrm{and}}%
\ \ Z_{\overline{j}}=\frac{1}{2}(e_{j}+ie_{\widetilde{j}})
\]
with $\widetilde{j}=n+j,\ j=1,...n.$ The operator that we are interested in
this paper will be
\[
L=\sum_{j=1}^{n}(e_{j}{}^{2}+e_{\widetilde{j}}{}^{2})=2\text{\ }\Delta_{b}{.}%
\]

\begin{definition}
Let $M$ be a smooth connected manifold with a positive measure and vector
fields $\{e_{i},Y_{\alpha}\}_{i\in I_{d},\alpha\in\Lambda}$ spanning the
tangent space $TM$. For $\rho_{1}\in%
\mathbb{R}
,\ \rho_{2}>0,\ \kappa\geq0,\ m>0,$ we say that $M$ satisfies the generalized
curvature-dimension inequality $CD(\rho_{1},\rho_{2},\kappa,m)$ if
\begin{equation}
\frac{1}{m}(Lf)^{2}+(\rho_{1}-\frac{\kappa}{\nu})\Gamma(f,\ f)+\rho_{2}%
\Gamma^{Z}(f,\ f)\leq\Gamma_{2}(f,\ f)+\nu\Gamma_{2}^{Z}(f,\ f)\label{2016A}%
\end{equation}
for any smooth function $f\in C^{\infty}(M)$ and $\nu>0.$ Here
\[%
\begin{array}
[c]{ccl}%
\Gamma(f,\ f) & := &
{\displaystyle\sum\limits_{j\in I_{d}}}
\left\vert e_{j}f\right\vert ^{2},\\
\Gamma^{Z}(f,\ f) & := &
{\displaystyle\sum\limits_{\alpha\in I_{h}}}
\left\vert Y_{\alpha}f\right\vert ^{2},\\
\Gamma_{2}(f,\ f) & := & \frac{1}{2}[L(\Gamma(f,\ f))-2%
{\displaystyle\sum\limits_{j\in I_{d}}}
(e_{j}f)(e_{j}Lf)],\\
\Gamma_{2}^{Z}(f,\ f) & := & \frac{1}{2}[L(\Gamma^{Z}(f,\ f))-2%
{\displaystyle\sum\limits_{\alpha\in I_{h}}}
(Y_{\alpha}f)(Y_{\alpha}Lf)].
\end{array}
\]
Note that we also have
\[
\Gamma_{2}(f,\ f)=%
{\displaystyle\sum\limits_{i,j\in I_{d}}}
\left\vert e_{i}e_{j}f\right\vert ^{2}+%
{\displaystyle\sum\limits_{j\in I_{d}}}
(e_{j}f)([L,e_{j}]f)
\]
and
\[
\Gamma_{2}^{Z}(f,\ f)=%
{\displaystyle\sum\limits_{i\in I_{d},\alpha\in I_{h}}}
\left\vert e_{i}Y_{\alpha}f\right\vert ^{2}+%
{\displaystyle\sum\limits_{\alpha\in I_{h}}}
(Y_{\alpha}f)([L,Y_{\alpha}]f).
\]

\end{definition}

In Lemma \ref{l2}, we will derive a curvature-dimension inequality
(\ref{2016A}) in a closed pseudohermitian manifold of the nonvanishing torsion
tensor. As a result, we are able to obtain the following CR Li-Yau gradient
estimate which is served as a generalization of the CR Li-Yau gradient
estimate in a closed pseudohermitian $(2n+1)$-manifold with nonnegative
pseudohermitian Ricci curvature and vanishing torsion as in \cite{ckl},
\cite{ckl1} and \cite{bg}.

\begin{theorem}
\label{t1} Let $(M,J,\theta)$ be a closed pseudohermitian $\left(
2n+1\right)  $-manifold with
\[
\left(  2Ric-\left(  n-2\right)  Tor\right)  \left(  Z,Z\right)
\geq-k\left\langle Z,Z\right\rangle
\]
and%

\[
\max_{i,j\in I_{n}}\left\vert A_{ij}\right\vert \leq\overline{A}%
,\ \ \ \max_{i,j\in I_{n}}\left\vert A_{ij,\overline{i}}\right\vert
\leq\overline{B}%
\]
for $Z\in\Gamma\left(  T_{1,0}M\right)  $, $k\geq0$ and $\overline
{A},\ \overline{B}$ as positive constants. Suppose that $u\left(  x,t\right)
$ is the positive solution of (\ref{heat eq}) on $M\times\left[  0,\text{
}\infty\right)  .$ Then there exist $\delta_{0}=\delta_{0}(n,k,\overline
{A},\overline{B})>>1$ such that $f\left(  x,t\right)  =\ln u\left(
x,t\right)  $ satisfies the following gradient estimate
\begin{equation}
\left\vert \nabla_{b}f\right\vert ^{2}-\delta f_{t}<\frac{C_{1}}{t}%
+C_{2}\label{2}%
\end{equation}
for $\delta\geq\delta_{0}$ and
\[%
\begin{array}
[c]{ccl}%
C_{1} & = & \frac{1}{2}\max\left\{  n\left(  n+1\right)  \delta^{2}%
+\frac{8\sqrt{3}\left(  n+1\right)  ^{2}\delta^{2}}{\left(  \delta-\delta
_{0}\right)  },\frac{3n\left(  n+1\right)  \delta^{2}}{4\left(  \delta
-\delta_{0}\right)  ^{2}}\left[  \left(  k+\frac{\overline{B}^{2}}{2\left(
n+1\right)  }\right)  \frac{\left(  \delta-\delta_{0}\right)  }{2n\left(
n+1\right)  \overline{A}\delta}+\frac{16\left(  n+1\right)  }{n}\right]
^{2}\right\}  .\\
C_{2} & = & \frac{1}{2}\max\{\left(  k+\frac{\overline{B}^{2}}{2\left(
n+1\right)  }\right)  \frac{\sqrt{3}n\left(  n+1\right)  \delta^{2}}{2\left(
\delta-\delta_{0}\right)  }+16\sqrt{3}\left(  n+1\right)  ^{2}\frac{\delta
^{3}\overline{A}}{\left(  \delta-\delta_{0}\right)  ^{2}},\\
&  & \frac{3\left(  n+1\right)  \delta}{8n\overline{A}\left(  \delta
-\delta_{0}\right)  }\left(  k+\frac{\overline{B}^{2}}{2\left(  n+1\right)
}+\frac{32n\left(  n+1\right)  \delta\overline{A}}{\left(  \delta-\delta
_{0}\right)  }\right)  ^{2}\}.
\end{array}
\]

\end{theorem}

As a consequence, we have $C_{2}=0$ if $k=0$ and $\overline{A}=0.$ Hence, we have

\begin{corollary}
\label{C1} Let $\left(  M,\text{ }J,\text{ }\theta\right)  $ be a closed
pseudohermitian $(2n+1)$-manifold with nonnegative pseudohermitian Ricci
curvature and vanishing torsion. If $u\left(  x,t\right)  $ is the positive
solution of (\ref{heat eq}) on $M\times\left[  0,\text{ }\infty\right)  $.
Then $f\left(  x,t\right)  =\ln u\left(  x,t\right)  $ satisfies the following
gradient estimate
\begin{equation}
\left\vert \nabla_{b}f\right\vert ^{2}-\delta f_{t}<\frac{C_{1}}%
{t}.\label{2011a}%
\end{equation}

\end{corollary}

\begin{remark}
In fact, in \cite{ckl1}, we get the following CR Li-Yau gradient estimate in a
closed pseudohermitian $(2n+1)$-manifold with nonnegative pseudohermitian
Ricci curvature and vanishing torsion. That is
\[
\left\vert \nabla_{b}f\right\vert ^{2}-(1+\frac{3}{n})f_{t}+\frac{n}{3}%
t(f_{0})^{2}<\frac{(\frac{9}{n}+6+n)}{t},
\]
but where we can not deal with the case of nonvanishing torsion tensors. The
major different here is : we apply the generalized curvature-dimension
inequality, which holds as in Lemma \ref{l2}, and Cao-Yau's method (\cite{cy})
to derive the gradient estimate in a closed pseudohermitian $(2n+1)$ -manifold
with nonvanishing torsion tensors.
\end{remark}

Next we have the CR version of Li-Yau Harnack inequality and upper bound
estimate for the heat kernel as in \cite{cftw} and \cite{cy}.

\begin{theorem}
\label{t2} Under the same hypothesis of Theorem \ref{t1}, suppose that $u$ is
the positive solution of
\[
(\Delta_{b}-\frac{\partial}{\partial t})u=0
\]
on $M\times\lbrack0,+\infty)$. Then for any $x_{1},x_{2}\in M$ and
$0<t_{1}<t_{2}<+\infty,$ there exists a constant $\delta_{0}(n,k,\overline{A
},\overline{B})>1$ such that
\[
\frac{u(x_{1},t_{1})}{u(x_{2},t_{2})}\leq\left(  \frac{t_{2}}{t_{1}}\right)
^{\frac{C_{1}^{\prime}(n,\delta)}{\delta}}\exp\left(  \frac{\delta}{4}
\frac{d_{cc}(x_{1},x_{2})^{2}}{t_{2}-t_{1}}+\frac{C_{2}^{\prime}%
(n,k,\delta,\overline{A},\overline{B})}{\delta}(t_{2}-t_{1})\right)
\]
for $\delta\geq\delta_{0}(n,k,\overline{A},\overline{B}).$ Here we denote the
Carnot-Carath\'{e}odory distance in $(M,J,\theta)$ by $d_{cc}$.
\end{theorem}

\begin{theorem}
\label{t4} Under the same hypothesis of Theorem \ref{t1}, suppose that
$H(x,y,t)$ is the heat kernel of
\[
(\Delta_{b}-\frac{\partial}{\partial t})u=0
\]
on $M\times\lbrack0,+\infty)$. Then there exists a constant $\delta_{1}>0 $
such that
\[
H(x,y,t)\leq C(\varepsilon)^{\delta_{1}}\frac{1}{\sqrt{vol\left(  B_{x}(
\sqrt{t})\right)  vol\left(  B_{y}(\sqrt{t})\right)  }}\exp\left(  \frac{
C_{2}^{\prime}(n,k,\delta,\overline{A},\overline{B})}{\delta}\varepsilon
t-\frac{d_{cc}(x,y)^{2}}{(4+\varepsilon)t}\right)
\]
for $\varepsilon\in(0,1)$ and $C(\varepsilon)\rightarrow+\infty$ as
$\varepsilon\rightarrow0^{+}$.
\end{theorem}

In the Cao-Yau gradient estimate for a positive solution of an operator with
respect to the sum of squares of vector fields of step $2$, the key estimates
are $(2.10),\ (2.12)\ $and $(2.14)$ of (\cite{cy}). This in fact, resembles
the generalized curvature-dimension inequality (\ref{2016A}) with $\ $some
certain $\rho_{1},\rho_{2},\kappa$ and $m.$ However this is not the case for
step $3$ and up. Then, as in Theorem \ref{t3}, it was an important insight
that one can use the generalized curvature-dimension inequality as a
substitute for the lower Ricci curvature bound on spaces where a direct
generalization of Ricci curvature is not available.

We start to setup the Li-Yau gradient estimate for a positive solution of an
operator with respect to the sum of squares of vector fields of higher step.
For simplicity, we assume that $M$ is of step $3$, i.e.
\begin{equation}
\left[  e_{i},\left[  e_{j},\left[  e_{k},e_{l}\right]  \right]  \right]
=a_{ijkl}^{n}e_{n}+b_{ijkl}^{\eta}Y_{\eta}^{\prime}+c_{ijkl}^{A}Y_{A}%
^{\prime\prime}\label{2016AAA}%
\end{equation}
for $a_{ijkl}^{n},b_{ijkl}^{\eta},c_{ijkl}^{A}\in C^{\infty}\left(  M\right)
$ with $\left\{  Y_{\alpha}\right\}  _{\alpha\in\Lambda}:=\left\{  Y_{\eta
}^{\prime}=\left[  e_{i},e_{j}\right]  \right\}  _{i,j\in I_{d}}\cup\left\{
Y_{A}^{\prime\prime}=\left[  e_{i,}\left[  e_{j},e_{k}\right]  \right]
\right\}  _{i,j,k\in I_{d}}$. We denote the supremum of coefficients as:
\[
\left\{
\begin{array}
[c]{lllll}%
a=\sup\left\vert a_{ijkl}^{n}\right\vert  & , & b=\sup\left\vert
b_{ijkl}^{\eta}\right\vert  & , & c=\sup\left\vert c_{ijkl}^{A}\right\vert ,\\
a^{\prime}=\sup\left\vert e_{h}a_{ijkl}^{n}\right\vert  & , & b^{\prime}%
=\sup\left\vert e_{h}b_{ijkl}^{\eta}\right\vert  & , & c^{\prime}%
=\sup\left\vert e_{h}c_{ijkl}^{A}\right\vert .
\end{array}
\right.
\]

\begin{theorem}
\label{t3} Let $M$ be a smooth connected manifold with a positive measure
satisfying the generalized curvature-dimension inequality $CD(\rho_{1}%
,\rho_{2},\kappa,m)$ and let $L$ be an operator with respect to the sum of
squares of vector fields $\{e_{1},\ e_{2},\ ...,e_{d}\}$ satisfying the
condition (\ref{2016AAA}). Suppose that $u$ is the positive solution of
\begin{equation}
(L-\frac{\partial}{\partial t})u=0\label{800}%
\end{equation}
on $M\times\lbrack0,+\infty).$ Then$\ $ for all $\ \frac{1}{2}<\lambda
<\frac{2}{3}$, there exists $\delta_{0}=\delta_{0}\left(  \lambda,\rho
_{1},\rho_{2},\kappa,d,h\right)  >1$ such that for any $\delta>\delta_{0}$
\[%
{\displaystyle\sum\limits_{j\in I_{d}}}
\frac{\left\vert e_{j}u\right\vert ^{2}}{u^{2}}+%
{\displaystyle\sum\limits_{\alpha\in\Lambda}}
\left(  1+\frac{\left\vert Y_{\alpha}u\right\vert ^{2}}{u^{2}}\right)
^{\lambda}-\delta\frac{u_{t}}{u}\leq\frac{C_{1}}{t}+C_{2}+C_{3}t^{\frac
{\lambda}{\lambda-1}},
\]
where $C_{1},C_{2},C_{3}$ are all positive constants depending on
$d,\lambda,\delta,a,a^{\prime},b,b^{\prime},c,c^{\prime},\rho_{1},\rho
_{2},\kappa,m.$
\end{theorem}

\begin{remark}
1. In the paper of \cite{bg}, they proved the $L^{p}$ version of Li-Yau type
gradient estimates for $2\leq p\leq\infty$ under the assumption of the
generalized curvature-dimension inequality via the semigroup method in the
sub-Riemannian geometry setting.

2. We can obtain the Li-Yau Harnack inequality and upper bound estimate for
the heat kernel of $L-\frac{\partial}{\partial t}$ with respect to the sum of
squares of vector fields as in \cite{cy}. We also refer to \cite{js},
\cite{ks1}, \cite{ks2} and \cite{m} for some details along this direction.
\end{remark}

We briefly describe the methods used in our proofs. In section $3$, we derive
a generalized curvature-dimension inequality in a closed pseudohermitian
$(2n+1)$-manifold. In order to gain insight for the estimate, we first derive
the CR Li-Yau gradient estimate and the Harnack inequality for the CR heat
equation in a closed pseudohermitian manifold as in section $4$. Then, for
simplicity, we will derive the Li-Yau gradient estimate for the sum of squares
of vector fields of step three as in section $5$. Similar estimates will hold
for the sum of squares of vector fields of higher step as well.

\textbf{Acknowledgement } The authors would like to express their profound
gratitude to Prof. S.-T. Yau for bringing this project to them and his
inspirations of the Li-Yau gradient estimate for the sum of squares of vector fields.

\section{Preliminary}

We introduce some basic materials about a pseudohermitian manifold (see
\cite{dt} , \cite{ckl}, and \cite{l} for more details). Let $(M,\xi)$ be a
$(2n+1)$-dimensional, orientable, contact manifold with contact structure
$\xi$. A CR structure compatible with $\xi$ is an endomorphism $J:\xi
\rightarrow\xi$ such that $J^{2}=-1$. We also assume that $J$ satisfies the
integrability condition: If $X$ and $Y$ are in $\xi$, then so are
$[JX,Y]+[X,JY]$ and $J([JX,Y]+[X,JY])=[JX,JY]-[X,Y]$.

Let $\left\{  \mathbf{T},Z_{\alpha},Z_{\bar{\alpha}}\right\}  $ be a frame of
$TM\otimes\mathbb{C}$, where $Z_{\alpha}$ is any local frame of $T_{1,0}%
,\ Z_{\bar{\alpha}}=\overline{Z_{\alpha}}\in T_{0,1}$ and $\mathbf{T} $ is the
characteristic vector field. Then $\left\{  \theta,\theta^{\alpha}%
,\theta^{\bar{\alpha}}\right\}  $, the coframe dual to $\left\{
\mathbf{T},Z_{\alpha},Z_{\bar{\alpha}}\right\}  $, satisfies
\begin{equation}
d\theta=ih_{\alpha\overline{\beta}}\theta^{\alpha}\wedge\theta^{\overline
{\beta}}\label{72}%
\end{equation}
for some positive definite hermitian matrix of functions $(h_{\alpha\bar
{\beta}})$. If we have this contact structure, we also call such $M$ a
strictly pseudoconvex CR $(2n+1)$-manifold.

The Levi form $\left\langle \ ,\ \right\rangle _{L_{\theta}}$ is the Hermitian
form on $T_{1,0}$ defined by%
\[
\left\langle Z,W\right\rangle _{L_{\theta}}=-i\left\langle d\theta
,Z\wedge\overline{W}\right\rangle .
\]
We can extend $\left\langle \ ,\ \right\rangle _{L_{\theta}}$ to $T_{0,1}$ by
defining $\left\langle \overline{Z},\overline{W}\right\rangle _{L_{\theta}%
}=\overline{\left\langle Z,W\right\rangle }_{L_{\theta}}$ for all $Z,W\in
T_{1,0}$. The Levi form induces naturally a Hermitian form on the dual bundle
of $T_{1,0}$, denoted by $\left\langle \ ,\ \right\rangle _{L_{\theta}^{\ast}%
}$, and hence on all the induced tensor bundles. Integrating the Hermitian
form (when acting on sections) over $M$ with respect to the volume form
$d\mu=\theta\wedge(d\theta)^{n}$, we get an inner product on the space of
sections of each tensor bundle.

The pseudohermitian connection of $(J,\theta)$ is the connection $\nabla$ on
$TM\otimes\mathbb{C}$ (and extended to tensors) given in terms of a local
frame $Z_{\alpha}\in T_{1,0}$ by%

\[
\nabla Z_{\alpha}=\omega_{\alpha}{}^{\beta}\otimes Z_{\beta},\quad\nabla
Z_{\bar{\alpha}}=\omega_{\bar{\alpha}}{}^{\bar{\beta}}\otimes Z_{\bar{\beta}%
},\quad\nabla T=0,
\]
where $\omega_{\alpha}{}^{\beta}$ are the $1$-forms uniquely determined by the
following equations :%

\[%
\begin{split}
d\theta^{\beta} &  =\theta^{\alpha}\wedge\omega_{\alpha}{}^{\beta}%
+\theta\wedge\tau^{\beta},\\
0 &  =\tau_{\alpha}\wedge\theta^{\alpha},\\
0 &  =\omega_{\alpha}{}^{\beta}+\omega_{\bar{\beta}}{}^{\bar{\alpha}},
\end{split}
\]
We can write (by Cartan lemma) $\tau_{\alpha}=A_{\alpha\gamma}\theta^{\gamma}$
with $A_{\alpha\gamma}=A_{\gamma\alpha}$. The curvature of Tanaka-Webster
connection, expressed in terms of the coframe $\{ \theta=\theta^{0}%
,\theta^{\alpha},\theta^{\bar{\alpha}}\}$, is
\[%
\begin{split}
\Pi_{\beta}{}^{\alpha} &  =\overline{\Pi_{\bar{\beta}}{}^{\bar{\alpha}}%
}=d\omega_{\beta}{}^{\alpha}-\omega_{\beta}{}^{\gamma}\wedge\omega_{\gamma}%
{}^{\alpha},\\
\Pi_{0}{}^{\alpha} &  =\Pi_{\alpha}{}^{0}=\Pi_{0}{}^{\bar{\beta}}=\Pi
_{\bar{\beta}}{}^{0}=\Pi_{0}{}^{0}=0.
\end{split}
\]
Webster showed that $\Pi_{\beta}{}^{\alpha}$ can be written
\[
\Pi_{\beta}{}^{\alpha}=R_{\beta}{}^{\alpha}{}_{\rho\bar{\sigma}}\theta^{\rho
}\wedge\theta^{\bar{\sigma}}+W_{\beta}{}^{\alpha}{}_{\rho}\theta^{\rho}%
\wedge\theta-W^{\alpha}{}_{\beta\bar{\rho}}\theta^{\bar{\rho}}\wedge
\theta+i\theta_{\beta}\wedge\tau^{\alpha}-i\tau_{\beta}\wedge\theta^{\alpha}%
\]
where the coefficients satisfy
\[
R_{\beta\bar{\alpha}\rho\bar{\sigma}}=\overline{R_{\alpha\bar{\beta}\sigma
\bar{\rho}}}=R_{\bar{\alpha}\beta\bar{\sigma}\rho}=R_{\rho\bar{\alpha}%
\beta\bar{\sigma}},\ \ \ W_{\beta\bar{\alpha}\gamma}=W_{\gamma\bar{\alpha
}\beta}.
\]
Here $R_{\gamma}{}^{\delta}{}_{\alpha\bar{\beta}}$ is the pseudohermitian
curvature tensor, $R_{\alpha\bar{\beta}}=R_{\gamma}{}^{\gamma}{}_{\alpha
\bar{\beta}}$ is the pseudohermitian Ricci curvature tensor and $A_{\alpha
\beta}$\ is the pseudohermitian torsion. Furthermore, we define the
bi-sectional curvature
\[
R_{\alpha\bar{\alpha}\beta\overline{\beta}}(X,Y)=R_{\alpha\bar{\alpha}%
\beta\overline{\beta}}X_{\alpha}X_{\overline{\alpha}}Y_{\beta}Y_{\bar{\beta}}%
\]
and the bi-torsion tensor
\[
T_{\alpha\overline{\beta}}(X,Y):=i(A_{\bar{\beta}\bar{\rho}}X^{\overline{\rho
}}Y_{\alpha}-A_{\alpha\rho}X^{\rho}Y_{\bar{\beta}})
\]
and the torsion tensor \
\[
Tor(X,Y):=h^{\alpha\bar{\beta}}T_{\alpha\overline{\beta}}(X,Y)=i(A_{\overline
{\alpha}\bar{\rho}}X^{\overline{\rho}}Y^{\overline{\alpha}}-A_{\alpha\rho
}X^{\rho}Y^{\alpha})
\]
for any $X=X^{\alpha}Z_{\alpha},\ Y=Y^{\alpha}Z_{\alpha}$ in $T_{1,0}.$

We will denote the components of the covariant derivatives with indices
preceded by a comma; thus write $A_{\alpha\beta,\gamma}$. The indices
$\{0,\alpha,\bar{\alpha}\}$ indicate derivatives with respect to
$\{T,Z_{\alpha},Z_{\bar{\alpha}}\}$. For derivatives of a scalar function, we
will often omit the comma, for instance, $u_{\alpha}=Z_{\alpha}u,\ u_{\alpha
\bar{\beta}}=Z_{\bar{\beta}}Z_{\alpha}u-\omega_{\alpha}{}^{\gamma}%
(Z_{\bar{\beta}})Z_{\gamma}u.$In particular,%

\[%
\begin{array}
[c]{c}%
|\nabla_{b}u|^{2}=2\sum_{\alpha}u_{\alpha}u_{\overline{\alpha}},\quad
|\nabla_{b}^{2}u|^{2}=2\sum_{\alpha,\beta}(u_{\alpha\beta}u_{\overline{\alpha
}\overline{\beta}}+u_{\alpha\overline{\beta}}u_{\overline{\alpha}\beta}).
\end{array}
\]

Also
\[%
\begin{array}
[c]{c}%
\Delta_{b}u=Tr\left(  (\nabla^{H})^{2}u\right)  =\sum_{\alpha}(u_{\alpha
\bar{\alpha}}+u_{\bar{\alpha}\alpha}).
\end{array}
\]

Next we recall the following commutation relations (\cite{l}).\ Let $\varphi$
be a scalar function and $\sigma=\sigma_{\alpha}\theta^{\alpha}$ be a $\left(
1,0\right)  $ form, $\varphi_{0}=\mathbf{T}\varphi,$ then we have%

\[%
\begin{array}
[c]{ccl}%
\varphi_{\alpha\beta} & = & \varphi_{\beta\alpha},\\
\varphi_{\alpha\bar{\beta}}-\varphi_{\bar{\beta}\alpha} & = & ih_{\alpha
\overline{\beta}}\varphi_{0},\\
\varphi_{0\alpha}-\varphi_{\alpha0} & = & A_{\alpha\beta}\varphi^{\beta},\\
\sigma_{\alpha,0\beta}-\sigma_{\alpha,\beta0} & = & \sigma_{\alpha,\bar
{\gamma}}A^{\overline{\gamma}}{}_{\beta}-\sigma^{\overline{\gamma}}%
A_{\alpha\beta,\bar{\gamma}},\\
\sigma_{\alpha,0\bar{\beta}}-\sigma_{\alpha,\bar{\beta}0} & = & \sigma
_{\alpha,\gamma}A^{\gamma}{}_{\bar{\beta}}+\sigma^{\overline{\gamma}}%
A_{\bar{\gamma}\bar{\beta},\alpha},
\end{array}
\]
and
\begin{equation}%
\begin{array}
[c]{cl}%
\left(  1\right)  & \varphi_{e_{j}e_{\widetilde{k}}}-\varphi_{e_{\widetilde
{k}}e_{j}}=2h_{j\overline{k}}\varphi_{0},\\
\left(  2\right)  & \varphi_{e_{j}e_{k}}-\varphi_{e_{k}e_{j}}=0,\\
\left(  3\right)  & \varphi_{0e_{j}}-\varphi_{e_{j}0}=\varphi_{e_{l}%
}\operatorname{Re}A_{j}^{\overline{l}}-\varphi_{e_{\widetilde{l}}%
}\operatorname{Im}A_{j}^{\overline{l}},\\
\left(  4\right)  & \varphi_{0e_{\widetilde{j}}}-\varphi_{e_{\widetilde{j}}%
0}=-\varphi_{e_{l}}\operatorname{Im}A_{j}^{\overline{l}}-\varphi
_{e\widetilde{l}}\operatorname{Re}A_{j}^{\overline{l}}.
\end{array}
\label{6}%
\end{equation}

Finally we introduce the concept about the Carnot-Carath\'{e}odory distance in
a closed pseudohermitian manifold. \ 

\begin{definition}
A piecewise smooth curve $\gamma:[0,1]\rightarrow M$ is said to be horizontal
if $\gamma\ ^{\prime}(t)\in\xi$ whenever $\gamma\ ^{\prime}(t)$ exists. The
length of $\gamma$ is then defined by
\[
l(\gamma)=\int_{0}^{1}\left\langle \gamma\ ^{\prime}(t),\gamma\ ^{\prime
}(t)\right\rangle _{L_{\theta}}^{\frac{1}{2}}dt.
\]
The Carnot-Carath\'{e}odory distance between two points $p,\ q\in M$ is
\[
d_{cc}(p,q)=\inf\left\{  l(\gamma)|\ \gamma\in C_{p,q}\right\}  ,
\]
where $C_{p,q}$ is the set of all horizontal curves joining $p$ and $q$. By
Chow connectivity theorem \cite{cho}, there always exists a horizontal curve
joining $p$ and $q$, so the distance is finite. The diameter $d_{c}$ is
defined by
\[
d_{c}(M)=\sup\left\{  d_{c}(p,q)|\ p,q\in M\right\}  .
\]
Note that there is a minimizing geodesic joining $p$ and $q$ so that its
length is equal to the distance $d_{cc}(p,q).$
\end{definition}

\section{A Generalized Curvature-Dimension Inequality}

Now we proceed to derive a curvature-dimension inequality in a closed
pseudohermitian $(2n+1)$-manifold under the specific assumptions on the
pseudohermitian Ricci curvature tensor and the torsion tensor. In particular,
in the case of vanishing torsion tensors, we have the following lemma.

\begin{lemma}
\label{l1} If $(M,J,\theta)$ is a pseudohermitian $(2n+1)$-manifold of
vanishing torsion with
\begin{equation}
2Ric\left(  Z,Z\right)  \geq-k\left\langle Z,Z\right\rangle \label{17}%
\end{equation}
for $Z\in\Gamma\left(  T_{1,0}M\right)  $, $k\geq0$, then $M$ satisfies the
curvature-dimension inequality $CD(-k,2n,4,2n)$.

\begin{proof}
By the CR Bochner formulae (see \cite{g})%
\[
\frac{1}{2}\Delta_{b}\left\vert \nabla_{b}f\right\vert ^{2}=\left\vert
Hess(f)\right\vert ^{2}+\left\langle \nabla_{b}f,\nabla_{b}(\Delta
_{b}f)\right\rangle +(2Ric-\left(  n-2\right)  Tor)((\nabla_{b}f)_{c}%
,(\nabla_{b}f)_{c})+2\left\langle J\nabla_{b}f,\nabla_{b}f_{0}\right\rangle ,
\]
where $(\nabla_{b}f)_{c}$ is the $T_{1,0}M$-component of $(\nabla_{b}f)$, we
have%
\[
\Gamma_{2}(f,f)=\left\vert Hess(f)\right\vert ^{2}+(2Ric-\left(  n-2\right)
Tor)((\nabla_{b}f)_{c},(\nabla_{b}f)_{c})+2\left\langle J\nabla_{b}%
f,\nabla_{b}f_{0}\right\rangle \text{.}%
\]
With the equality
\[
\Gamma_{2}^{Z}(f,f)=\left\vert \nabla_{b}f_{0}\right\vert ^{2}+f_{0}%
[\Delta_{b},T]f,
\]
we have
\begin{equation}%
\begin{array}
[c]{ccc}%
\Gamma_{2}(f,f)+\nu\Gamma_{2}^{Z}(f,f) & = & 4[\left\vert Hess(f)\right\vert
^{2}+(2Ric-\left(  n-2\right)  Tor)((\nabla_{b}f)_{c},(\nabla_{b}f)_{c})\\
&  & +2\left\langle J\nabla_{b}f,\nabla_{b}f_{0}\right\rangle ]+2\nu\left\vert
\nabla_{b}f_{0}\right\vert ^{2}+2\nu f_{0}[\Delta_{b},T]f.
\end{array}
\label{3}%
\end{equation}
On the other hand, we have
\begin{equation}
\left\vert Hess(f)\right\vert ^{2}=2(%
{\displaystyle\sum\limits_{i,j\in I_{n}}}
\left\vert f_{ij}\right\vert ^{2}+%
{\displaystyle\sum\limits_{i,j\in I_{n}}}
\left\vert f_{i\overline{j}}\right\vert ^{2})\geq\frac{1}{2n}\left\vert
\Delta_{b}f\right\vert ^{2}+\frac{n}{2}\left\vert f_{0}\right\vert
^{2}\label{18}%
\end{equation}
and
\begin{equation}
\left\langle J\nabla_{b}f,\nabla_{b}f_{0}\right\rangle \geq-\frac{\left\vert
\nabla_{b}f\right\vert ^{2}}{\nu}-\frac{\nu}{4}\left\vert \nabla_{b}%
f_{0}\right\vert ^{2}.\label{4}%
\end{equation}
Now it follows from (\ref{3}), (\ref{18}), (\ref{4}) and curvature assumptions%

\begin{equation}%
\begin{array}
[c]{ccl}%
\Gamma_{2}(f,f)+\nu\Gamma_{2}^{Z}(f,f) & \geq & \frac{2}{n}(\left\vert
\Delta_{b}f\right\vert ^{2}+2n\left\vert f_{0}\right\vert ^{2})+4(2Ric-\left(
n-2\right)  Tor)((\nabla_{b}f)_{c},(\nabla_{b}f)_{c})\\
&  & -8\frac{\left\vert \nabla_{b}f\right\vert ^{2}}{\nu}+2\nu f_{0}%
[\Delta_{b},T]f\\
& \geq & \frac{2}{n}\left\vert \Delta_{b}f\right\vert ^{2}+\left(
-2k-\frac{8}{\nu}\right)  \left\vert \nabla_{b}f\right\vert ^{2}+2n\left\vert
f_{0}\right\vert ^{2}+2\nu f_{0}[\Delta_{b},T]f.
\end{array}
\label{4ab}%
\end{equation}
Finally, it follows from the commutation relation (\cite{ckl}) that
\begin{equation}
\Delta_{b}f_{0}=\left(  \Delta_{b}f\right)  _{0}+2[\left(  A_{\alpha\beta
}f^{\alpha}\right)  ^{\beta}+(A_{\overline{\alpha}\overline{\beta}%
}f^{\overline{\alpha}})^{\overline{\beta}}].\label{4a}%
\end{equation}
But $A_{\alpha\beta}=0,$ hence
\[
\left[  \Delta_{b},T\right]  f=0.
\]
All these imply%
\[
\Gamma_{2}(f,f)+\nu\Gamma_{2}^{Z}(f,f)\geq\frac{2}{n}\left\vert \Delta
_{b}f\right\vert ^{2}+\left(  -2k-\frac{8}{\nu}\right)  \left\vert \nabla
_{b}f\right\vert ^{2}+2n\left\vert f_{0}\right\vert ^{2}.
\]

\end{proof}
\end{lemma}

\begin{remark}
In a closed pseudohermitian $\left(  2n+1\right)  $-manifold of vanishing
torsion tensors, the CR Bochner formulae (\ref{2015}) is equivalent to the
curvature-dimension inequality (\ref{2016A}) which also observed in the paper
of \cite{bg}.
\end{remark}

As for the curvature-dimension inequality in a closed pseudohermitian $\left(
2n+1\right)  $-manifold of nonvanishing torsion tensors, we have

\begin{lemma}
\label{l2} Let $(M,J,\theta)$ be a closed pseudohermitian $\left(
2n+1\right)  $-manifold of
\[
\left(  2Ric-\left(  n-2\right)  Tor\right)  \left(  Z,Z\right)
\geq-k\left\langle Z,Z\right\rangle
\]
for $Z\in\Gamma\left(  T_{1,0}M\right)  $, $k\geq0$ and%
\[
\max_{i,j\in I_{n}}\left\vert A_{ij}\right\vert \leq\overline{A}%
,\ \ \max_{i,j\in I_{n}}\left\vert A_{ij,\overline{i}}\right\vert
\leq\overline{B}%
\]
for nonnegative constants $\overline{A},\ \overline{B}$, Then $M$ satisfies
the curvature-dimension inequality $CD(-k-2nN\varepsilon_{1}\overline{B}%
^{2},\frac{2n}{m}-\frac{2n^{2}N}{\varepsilon_{1}}-\frac{2mn^{2}N^{2}%
\overline{A}^{2}}{m-1},4,2mn)$ for $1<m<+\infty$, $0<\varepsilon_{1}<+\infty$
and smaller $N>0$ such that
\[
\left(  \frac{2n}{m}-\frac{2n^{2}N}{\varepsilon_{1}}-\frac{2mn^{2}%
N^{2}\overline{A}^{2}}{m-1}\right)  >0
\]
and $0<\nu\leq N$.
\end{lemma}

\begin{proof}
It follows from (\ref{3}), (\ref{4}) and (\ref{4a}) that
\[%
\begin{array}
[c]{ccl}%
\Gamma_{2}(f,\ f)+\nu\Gamma_{2}^{Z}(f,\ f) & \geq & 8\left[
{\displaystyle\sum\limits_{\alpha,\beta}}
\left(  \left\vert f_{\alpha\beta}\right\vert ^{2}+\left\vert f_{\alpha
\overline{\beta}}\right\vert ^{2}\right)  \right]  -\left(  2k+\frac{8}{\nu
}\right)  \left\vert \nabla_{b}f\right\vert ^{2}\\
&  & -8\nu\left\vert f_{0}\right\vert
{\displaystyle\sum\limits_{\alpha,\beta}}
\left\vert (A_{\overline{\alpha}\overline{\beta},\alpha}f_{\beta}%
+A_{\overline{\alpha}\overline{\beta}}f_{\beta\alpha})\right\vert .
\end{array}
\]
Note that by using the Young inequality%
\begin{equation}
\left\vert f_{0}\right\vert \left(  \left\vert A_{\overline{\alpha}%
\overline{\beta},\alpha}f_{\beta}\right\vert +\left\vert A_{\overline{\alpha
}\overline{\beta}}f_{\beta\alpha}\right\vert \right)  \leq\frac{\left\vert
f_{0}\right\vert ^{2}}{4\varepsilon_{1}}+\varepsilon_{1}\left\vert
A_{\overline{\alpha}\overline{\beta,}\alpha}f_{\beta}\right\vert ^{2}%
+\frac{\left\vert f_{0}\right\vert ^{2}}{4\varepsilon_{2}}+\varepsilon
_{2}\left\vert A_{\overline{\alpha}\overline{\beta}}f_{\beta\alpha}\right\vert
^{2},\label{19}%
\end{equation}
for $\varepsilon_{1},\varepsilon_{2}>0$. Choose
\[
\varepsilon_{2}=\frac{m-1}{mN\overline{A}^{2}}%
\]
for $m>1$ and $N$ with $\nu\leq N$. This implies that $\left(  1-N\varepsilon
_{2}\overline{A}^{2}\right)  =\frac{1}{m}$.

It follows from (\ref{18}) that
\[%
\begin{array}
[c]{ccl}%
\Gamma_{2}(f,\ f)+\nu\Gamma_{2}^{Z}(f,\ f) & \geq & 8%
{\displaystyle\sum\limits_{\alpha,\beta}}
\left\vert f_{\alpha\overline{\beta}}\right\vert ^{2}+8%
{\displaystyle\sum\limits_{\alpha,\beta}}
\left(  1-\nu\varepsilon_{2}\left\vert A_{\overline{\alpha}\overline{\beta}%
}\right\vert ^{2}\right)  \left\vert f_{\beta\alpha}\right\vert ^{2}-\left(
2k+\frac{8}{\nu}\right)  \left\vert \nabla_{b}f\right\vert ^{2}\\
&  & -2\nu%
{\displaystyle\sum\limits_{\alpha,\beta}}
\left(  \frac{1}{\varepsilon_{1}}+\frac{1}{\varepsilon_{2}}\right)  \left\vert
f_{0}\right\vert ^{2}-8\nu\varepsilon_{1}%
{\displaystyle\sum\limits_{\alpha,\beta}}
\left\vert A_{\overline{\alpha}\overline{\beta},\alpha}f_{\beta}\right\vert
^{2}\\
& \geq & \frac{8}{m}%
{\displaystyle\sum\limits_{\alpha,\beta}}
\left(  \left\vert f_{\alpha\overline{\beta}}\right\vert ^{2}+\left\vert
f_{\beta\alpha}\right\vert ^{2}\right)  -\left(  2k+\frac{8}{\nu
}+4N\varepsilon_{1}n\overline{B}^{2}\right)  \left\vert \nabla_{b}f\right\vert
^{2}\\
&  & -2n^{2}N\left(  \frac{1}{\varepsilon_{1}}+\frac{1}{\varepsilon_{2}%
}\right)  \left\vert f_{0}\right\vert ^{2}\\
& \geq & \frac{4}{m}\left(  \frac{1}{2n}\left\vert \Delta_{b}f\right\vert
^{2}+\frac{n}{2}\left\vert f_{0}\right\vert ^{2}\right)  +\left(  2k+\frac
{8}{\nu}+4N\varepsilon_{1}n\overline{B}^{2}\right)  \left\vert \nabla
_{b}f\right\vert ^{2}\\
&  & -2n^{2}N\left(  \frac{1}{\varepsilon_{1}}+\frac{1}{\varepsilon_{2}%
}\right)  \left\vert f_{0}\right\vert ^{2}\\
& \geq & \frac{1}{2mn}\left(  Lf\right)  ^{2}+\left(  -k-2nN\varepsilon
_{1}\overline{B}^{2}-\frac{4}{\nu}\right)  \Gamma\left(  f,f\right) \\
&  & +\left(  \frac{2n}{m}-\frac{2n^{2}N}{\varepsilon_{1}}-\frac{2mn^{2}%
N^{2}\overline{A}^{2}}{m-1}\right)  \Gamma^{Z}\left(  f,f\right)  .
\end{array}
\]
Now we make $N$ smaller such that
\[
\frac{2n}{m}-\frac{2n^{2}N}{\varepsilon_{1}}-\frac{2mn^{2}N^{2}\overline
{A}^{2}}{m-1}>0.
\]

Then we are done.
\end{proof}

\begin{remark}
\bigskip By choosing $\overline{A}=0=\overline{B},$ $m\rightarrow
1^{+},\varepsilon_{1}\rightarrow+\infty$ and noting the inequality $\left(
\ref{19}\right)  $ in Lemma \ref{l2}, we are also able to have the same
conclusion in Lemma \ref{l1}.
\end{remark}

\section{The CR Li-Yau gradient estimate}

In this section, based on methods of \cite{cy} and \cite{ckl}, we first derive
the CR Li-Yau gradient estimate and the Harnack inequality for the CR heat
equation in a closed pseudohermitian manifold. Let $(M,J,\theta)$ be a closed
pseudohermitian $(2n+1)$-manifold and $u(x,t)$ be a positive solution of the
CR heat equation
\[
\left(  \Delta_{b}-\frac{\partial}{\partial t}\right)  u\left(  x,t\right)  =0
\]
on $M\times\left[  0,\text{ }\infty\right)  $. \ We denote that $f(x,t)=\ln
u(x,t)$. Modified by \cite{ckl}, we define a real-valued function
$F(x,t,\beta,\delta):M\times0,$\thinspace$T)\times\mathbf{R}^{+}%
\times\mathbf{R}^{+}\rightarrow\mathbf{R}$ by
\begin{equation}%
\begin{array}
[c]{c}%
F(x,t,\beta,\delta)=t\left[
{\displaystyle\sum\limits_{j\in I_{d}}}
\left\vert e_{j}f\right\vert ^{2}+\beta t%
{\displaystyle\sum\limits_{\alpha\in I_{h}}}
\left\vert Y_{\alpha}f\right\vert ^{2}-\delta f_{t}\right]
\end{array}
\label{2016}%
\end{equation}
for $x\in M,\ t\geq0,\ \beta>0,\ \delta>0$. Note that $\beta\rightarrow0^{+}$
if $T\rightarrow\infty$ as in the proof.

\begin{lemma}
\label{l3} Let $(M,J,\theta)$ be a closed pseudohermitian $(2n+1)$-manifold
and $u(x,t)$ be a positive solution of the CR heat equation
\[
\left(  L-\frac{\partial}{\partial t}\right)  u\left(  x,t\right)  =0
\]
on $M\times\left[  0,\text{ }\infty\right)  $. \ We have the identity%
\begin{equation}%
\begin{array}
[c]{ccl}%
(L-\frac{\partial}{\partial t})F & = & -\frac{F}{t}+2t[\Gamma_{2}(f,\ f)+\beta
t\Gamma_{2}^{Z}(f,\ f)]\\
&  & +4\beta t^{2}%
{\displaystyle\sum\limits_{j\in I_{d},\alpha\in I_{h}}}
(e_{j}f)(Y_{\alpha}f)([e_{j},Y_{\alpha}]f)\\
&  & -2%
{\displaystyle\sum\limits_{j\in I_{d}}}
(e_{j}f)(e_{j}F)-\beta t%
{\displaystyle\sum\limits_{\alpha\in I_{h}}}
\left\vert Y_{\alpha}f\right\vert ^{2}.
\end{array}
\label{5}%
\end{equation}

\begin{proof}
It follows from definitions of $\Gamma_{2}(f,\ f)$ and $\Gamma_{2}^{Z}(f,\ f)$
that
\[%
\begin{array}
[c]{ccl}%
LF & = & t\left[  L(\Gamma(f,\ f))+\beta tL(\Gamma^{Z}(f,\ f))-\delta
Lf_{t}\right] \\
& = & t\{[2\Gamma_{2}(f,\ f)+2%
{\displaystyle\sum\limits_{j\in I_{d}}}
\left(  e_{j}f\right)  \left(  e_{j}Lf\right)  ]\\
&  & +\beta t[2\Gamma_{2}^{Z}(f,\ f)+2%
{\displaystyle\sum\limits_{\alpha\in I_{h}}}
\left(  Y_{\alpha}f\right)  \left(  Y_{\alpha}Lf\right)  ]-\delta Lf_{t}\}.
\end{array}
\]
Then
\begin{equation}%
\begin{array}
[c]{ccl}%
(L-\frac{\partial}{\partial t})F & = & -\frac{F}{t}+t[2\Gamma_{2}%
(f,\ f)+2\beta t\Gamma_{2}^{Z}(f,\ f)+2%
{\displaystyle\sum\limits_{j}}
\left(  e_{j}f\right)  e_{j}\left(  L-\frac{\partial}{\partial t}\right)  f\\
&  & +2\beta t%
{\displaystyle\sum\limits_{\alpha}}
\left(  Y_{\alpha}f\right)  Y_{\alpha}\left(  L-\frac{\partial}{\partial
t}\right)  f-\beta%
{\displaystyle\sum\limits_{\alpha}}
\left(  Y_{\alpha}f\right)  ^{2}-\delta\frac{\partial}{\partial t}\left(
L-\frac{\partial}{\partial t}\right)  f].
\end{array}
\label{5a}%
\end{equation}
Since
\[
(L-\frac{\partial}{\partial t})f=-%
{\displaystyle\sum\limits_{j}}
\left\vert e_{j}f\right\vert ^{2}=-\frac{F}{t}+\beta t%
{\displaystyle\sum\limits_{\alpha}}
\left\vert Y_{\alpha}f\right\vert ^{2}-\delta f_{t},
\]
we obtain%

\begin{subequations}
\begin{equation}%
\begin{array}
[c]{l}%
2%
{\displaystyle\sum\limits_{j}}
\left(  e_{j}f\right)  e_{j}\left(  L-\frac{\partial}{\partial t}\right)
f+2\beta t%
{\displaystyle\sum\limits_{\alpha}}
\left(  Y_{\alpha}f\right)  Y_{\alpha}\left(  L-\frac{\partial}{\partial
t}\right)  f\\
-\beta%
{\displaystyle\sum\limits_{\alpha}}
\left(  Y_{\alpha}f\right)  ^{2}-\delta\frac{\partial}{\partial t}\left(
L-\frac{\partial}{\partial t}\right)  f\\
=2%
{\displaystyle\sum\limits_{j}}
\left(  e_{j}f\right)  e_{j}\left(  -\frac{F}{t}+\beta t%
{\displaystyle\sum\limits_{\alpha}}
\left\vert Y_{\alpha}f\right\vert ^{2}-\delta f_{t}\right) \\
+2\beta t%
{\displaystyle\sum\limits_{\alpha}}
\left(  Y_{\alpha}f\right)  Y_{\alpha}\left(  -%
{\displaystyle\sum\limits_{j}}
\left\vert e_{j}f\right\vert ^{2}\right)  -\beta%
{\displaystyle\sum\limits_{\alpha}}
\left(  Y_{\alpha}f\right)  ^{2}-\delta\frac{\partial}{\partial t}\left(
L-\frac{\partial}{\partial t}\right)  f\\
=2\beta t\left[
{\displaystyle\sum\limits_{j}}
\left(  e_{j}f\right)  e_{j}\left(
{\displaystyle\sum\limits_{\alpha}}
\left\vert Y_{\alpha}f\right\vert ^{2}\right)  +%
{\displaystyle\sum\limits_{\alpha}}
\left(  Y_{\alpha}f\right)  Y_{\alpha}\left(  -%
{\displaystyle\sum\limits_{j}}
\left\vert e_{j}f\right\vert ^{2}\right)  \right] \\
+2%
{\displaystyle\sum\limits_{j}}
\left(  e_{j}f\right)  e_{j}\left(  -\frac{F}{t}-\delta f_{t}\right)  -\beta%
{\displaystyle\sum\limits_{\alpha}}
\left(  Y_{\alpha}f\right)  ^{2}-\delta\frac{\partial}{\partial t}\left(  -%
{\displaystyle\sum\limits_{j}}
\left\vert e_{j}f\right\vert ^{2}\right) \\
=4\beta t%
{\displaystyle\sum\limits_{j,\alpha}}
(e_{j}f)(Y_{\alpha}f)([e_{j},Y_{\alpha}]f)-\frac{2}{t}%
{\displaystyle\sum\limits_{j}}
(e_{j}f)(e_{j}F)-\beta%
{\displaystyle\sum\limits_{\alpha}}
\left\vert Y_{\alpha}f\right\vert ^{2}\text{.}%
\end{array}
\label{5b}%
\end{equation}
Substitute (\ref{5b}) into (\ref{5a}), we have the identity (\ref{5}).
\end{subequations}
\end{proof}
\end{lemma}

\bigskip As a consequence of the identity (\ref{5}), we have proposition 4.2.

\begin{proposition}
If $M$ satisfies the curvature-dimension inequality $CD(\rho_{1},\rho
_{2},\kappa,m)$ for $\rho_{1}\in%
\mathbb{R}
,\rho_{2}>0,\kappa\geq0,m>0$, then
\begin{equation}%
\begin{array}
[c]{ccl}%
(L-\frac{\partial}{\partial t})F & \geq & -\frac{F}{t}+2t\left[  \frac{1}%
{m}(Lf)^{2}+(\rho_{1}-\frac{\kappa}{\beta t})\Gamma(f,f)+\rho_{2}\Gamma
^{Z}(f,f)\right]  +\\
&  & +4\beta t^{2}%
{\displaystyle\sum\limits_{j\in I_{d},\alpha\in I_{l}}}
(e_{j}f)(Y_{\alpha}f)([e_{j},Y_{\alpha}]f)\\
&  & -2%
{\displaystyle\sum\limits_{j\in I_{d}}}
(e_{j}f)(e_{j}F)-\beta t%
{\displaystyle\sum\limits_{\alpha\in I_{h}}}
\left\vert Y_{\alpha}f\right\vert ^{2}.
\end{array}
\label{7}%
\end{equation}

\end{proposition}

\bigskip

Now we proceed to prove\textbf{\ Theorem \ref{t1} : }

\begin{proof}
Note that $M$ satisfies the curvature-dimension inequality $CD(\rho_{1}%
,\rho_{2},\kappa,m)$ with $\rho_{1}<0,\rho_{2}>0,\kappa\geq0,m>0$ as in Lemma
\ref{l2}. Here we follow the method as in (\cite{cy}). Set%
\[
\left\{
\begin{array}
[c]{l}%
x=\left(  \delta_{0}\left\vert \nabla_{b}f\right\vert ^{2}-\delta
f_{t}\right)  \left(  x_{0},t_{0}\right)  \text{ for }\delta>\delta_{0}>2\\
\overline{x}=\left\vert \nabla_{b}f\right\vert ^{2}\left(  x_{0},t_{0}\right)
\\
y=\left\vert f_{0}\right\vert \left(  x_{0},t_{0}\right)
\end{array}
\right.
\]
where $\delta_{0}$ and $\left(  x_{0},t_{0}\right)  $ will be chosen later and
$f_{0}=Tf$ with $T:=Y_{\alpha}.$ From now on, $T$ denotes a positive real
number instead of a vector field.

If $F$ attains its maximum at $\left(  x_{0},t_{0}\right)  \in M\times
\lbrack0,T]$, then, by choosing a normal coordinate at $\left(  x_{0}%
,t_{0}\right)  $ and $\left(  \ref{6}\right)  $, $\left(  \ref{7}\right)  $
becomes
\begin{equation}
0\geq-\frac{F}{t_{0}}+2t_{0}\left[  \frac{1}{m}(2\Delta_{b}f)^{2}+2(\rho
_{1}-\frac{\kappa}{\beta t_{0}})\overline{x}+\rho_{2}y^{2}\right]  -16n\beta
t_{0}^{2}\overline{A}\overline{x}y-\beta t_{0}y^{2}\label{16}%
\end{equation}
for $d=2n$ and $h=1$. More precisely from the commutation relations $\left(
\ref{6}\right)  $, we have at $\left(  x_{0},t_{0}\right)  $
\[%
\begin{array}
[c]{cl}
&
{\displaystyle\sum\limits_{\alpha\in I_{2n}}}
\left(  e_{\alpha}f\right)  \left(  Tf\right)  \left(  \left[  e_{\alpha
},T\right]  f\right)  \left(  x_{0},t_{0}\right) \\
= & f_{0}%
{\displaystyle\sum\limits_{\alpha\in I_{2n}}}
f_{e_{\alpha}}\left(  e_{\alpha}Tf-Te_{\alpha}f\right) \\
= & f_{0}%
{\displaystyle\sum\limits_{\alpha\in I_{2n}}}
f_{e_{\alpha}}\left[  \left(  f_{0e_{\alpha}}+\left(  D_{e_{\alpha}}T\right)
f\right)  -\left(  f_{e_{\alpha}0}+\left(  D_{T}e_{\alpha}\right)  f\right)
\right] \\
= & f_{0}\left[
{\displaystyle\sum\limits_{j\in I_{n}}}
f_{e_{j}}\left(  f_{0e_{j}}-f_{e_{j}0}\right)  +%
{\displaystyle\sum\limits_{j\in I_{n}}}
f_{e_{\widetilde{j}}}\left(  f_{0e_{_{\widetilde{j}}}}-f_{e_{_{\widetilde{j}}%
}0}\right)  -%
{\displaystyle\sum\limits_{\alpha,\beta\in I_{2n}}}
\Gamma_{0e_{\alpha}}^{e_{\beta}}f_{e_{\alpha}}f_{e_{\beta}}\right] \\
= & f_{0}%
{\displaystyle\sum\limits_{j,l\in I_{n}}}
f_{e_{j}}\left(  f_{e_{l}}\operatorname{Re}A_{jl}-f_{e_{\widetilde{l}}%
}\operatorname{Im}A_{jl}\right) \\
& +f_{0}%
{\displaystyle\sum\limits_{j,l\in I_{n}}}
f_{e\widetilde{_{j}}}\left(  -f_{e_{l}}\operatorname{Im}A_{jl}%
-f_{e_{\widetilde{l}}}\operatorname{Re}A_{jl}\right)  -f_{0}%
{\displaystyle\sum\limits_{\alpha,\beta\in I_{2n}}}
\Gamma_{0e_{\alpha}}^{e_{\beta}}f_{e_{\alpha}}f_{e_{\beta}}\\
\geq & -\left\vert f_{0}\right\vert \overline{A}%
{\displaystyle\sum\limits_{j,l\in I_{n}}}
\left(  \left\vert f_{e_{j}}\right\vert +\left\vert f_{e\widetilde{_{j}}%
}\right\vert \right)  \left(  \left\vert f_{e_{l}}\right\vert +\left\vert
f_{e\widetilde{_{l}}}\right\vert \right) \\
\geq & -4n\overline{A}\overline{x}y.
\end{array}
\]
We divide the discussion into the following two cases :

$(I)$ $\ Case\ I$ $:x\geq0:$

By%
\[
(\Delta_{b}f)^{2}=\left(  f_{t}-\left\vert \nabla_{b}f\right\vert ^{2}\right)
^{2}=\left[  \frac{x}{\delta}+\left(  1-\frac{\delta_{0}}{\delta}\right)
\left\vert \nabla_{b}f\right\vert ^{2}\right]  ^{2}\geq\frac{x^{2}}{\delta
^{2}}+\frac{\left(  \delta-\delta_{0}\right)  ^{2}}{\delta^{2}}\left(
\left\vert \nabla_{b}f\right\vert ^{2}\right)  ^{2},
\]
we have%

\begin{equation}%
\begin{array}
[c]{ccl}%
0 & \geq & -\frac{F}{t_{0}}+\frac{8t_{0}}{m\delta^{2}}x^{2}+t_{0}\rho_{2}%
y^{2}\\
&  & +\frac{8t_{0}\left(  \delta-\delta_{0}\right)  ^{2}}{m\delta^{2}%
}\overline{x}^{2}+\left(  4t_{0}\rho_{1}-\frac{4\kappa}{\beta}\right)
\overline{x}\\
&  & +t_{0}\left(  \rho_{2}-\beta\right)  y^{2}-16n\beta t_{0}^{2}\overline
{A}\overline{x}y.
\end{array}
\label{15}%
\end{equation}
Let%
\[
\mathcal{A}:=t_{0}\left[  \frac{8\left(  \delta-\delta_{0}\right)  ^{2}%
}{m\delta^{2}}\overline{x}^{2}+\left(  4\rho_{1}-\frac{4\kappa}{\beta t_{0}%
}\right)  \overline{x}+\left(  \rho_{2}-\beta\right)  y^{2}-16n\beta
t_{0}\overline{A}\overline{x}y\right]
\]
and%
\[
A:=\frac{8\left(  \delta-\delta_{0}\right)  ^{2}}{m\delta^{2}}.
\]
We have
\[%
\begin{array}
[c]{ccl}%
\mathcal{A} & = & t_{0}\{A\left(  \overline{x}+\frac{2\rho_{1}-\frac{2\kappa
}{\beta t_{0}}-8n\beta t_{0}\overline{A}y}{A}\right)  ^{2}-\frac{\left(
2\rho_{1}-\frac{2\kappa}{\beta t_{0}}-8n\beta t_{0}\overline{A}y\right)  ^{2}%
}{A}+\left(  \rho_{2}-\beta\right)  y^{2}\}\\
& = & t_{0}\{A\left(  \overline{x}+\frac{2\rho_{1}-\frac{2\kappa}{\beta t_{0}%
}-8n\beta t_{0}\overline{A}y}{A}\right)  ^{2}+\left(  \rho_{2}-\beta
-\frac{64n^{2}\beta^{2}t_{0}^{2}\overline{A}^{2}}{A}\right)  y^{2}\\
&  & +\frac{32n\beta t_{0}\overline{A}}{A}\left(  \rho_{1}-\frac{\kappa}{\beta
t_{0}}\right)  y-4\frac{\left(  \rho_{1}-\frac{\kappa}{\beta t_{0}}\right)
^{2}}{A}\}.
\end{array}
\]
Choose
\[
\beta=\beta_{1}:=\min\{ \frac{\rho_{2}}{4},\frac{\sqrt{A\rho_{2}}%
}{16nT\overline{A}}\}.
\]
This implies that%
\[
B:=\left(  \rho_{2}-\beta-\frac{64n^{2}\beta^{2}t_{0}^{2}\overline{A}^{2}}%
{A}\right)  \geq\frac{\rho_{2}}{2}.
\]

$(i)$ Under the case
\[
T\geq T_{0}:=\frac{\sqrt{A\rho_{2}}}{4n\rho_{2}\overline{A}},
\]
we have
\[%
\begin{array}
[c]{ccl}%
\mathcal{A} & \geq & t_{0}\{By^{2}+\frac{32n\beta t_{0}\overline{A}}{A}\left(
\rho_{1}-\frac{\kappa}{\beta t_{0}}\right)  y-4\frac{\left(  \rho_{1}%
-\frac{\kappa}{\beta t_{0}}\right)  ^{2}}{A}\}\\
& = & t_{0}[B\left(  y+\frac{16n\beta t_{0}\overline{A}}{BA}\left(  \rho
_{1}-\frac{\kappa}{\beta t_{0}}\right)  \right)  ^{2}-\frac{16^{2}n^{2}%
\beta^{2}t_{0}^{2}\overline{A}^{2}+4BA}{BA^{2}}\left(  \rho_{1}-\frac{\kappa
}{\beta t_{0}}\right)  ^{2}\}\\
& \geq & -\frac{1}{A}\left(  \rho_{1}-\frac{16n\kappa\overline{A}}{\sqrt
{A\rho_{2}}}\frac{T}{t_{0}}\right)  ^{2}\left(  1+\frac{2t_{0}^{2}}{T^{2}%
}\right)  t_{0}\\
& \geq & -\frac{3}{A}\left(  \rho_{1}-\frac{16n\kappa\overline{A}}{\sqrt
{A\rho_{2}}}\frac{T}{t_{0}}\right)  ^{2}t_{0},
\end{array}
\]
Set%
\[
z=t_{0}x\text{.}%
\]

$(a)\ $If $x\geq\beta t_{0}\left\vert f_{0}\right\vert ^{2}$, \ it follows
from
\[
F=t_{0}\left[  \left(  2-\delta_{0}\right)  \left\vert \nabla_{b}f\right\vert
^{2}+x+\beta t_{0}\left\vert f_{0}\right\vert ^{2}\right]
\]
that $\left(  \ref{15}\right)  $ becomes%
\[
0\geq-2t_{0}x+\frac{8}{m\delta^{2}}\left(  t_{0}x\right)  ^{2}-\frac{3}%
{A}\left(  \rho_{1}t_{0}-\frac{16n\kappa\overline{A}}{\sqrt{A\rho_{2}}%
}T\right)  ^{2}%
\]
and then
\[
\left(  z-\frac{m\delta^{2}}{8}\right)  ^{2}\leq\frac{m\delta^{2}}{8}\left(
\frac{m\delta^{2}}{8}+\frac{3}{A}\left(  \rho_{1}t_{0}-\frac{16n\kappa
\overline{A}}{\sqrt{A\rho_{2}}}T\right)  ^{2}\right)  .
\]
Thus
\[
z\leq\frac{m\delta^{2}}{4}+\frac{\delta}{4}\sqrt{\frac{6m}{A}}\left\vert
\left(  \rho_{1}t_{0}-\frac{16n\kappa\overline{A}}{\sqrt{A\rho_{2}}}T\right)
\right\vert .
\]
This implies
\[
F\leq2t_{0}x\leq\frac{m\delta^{2}}{2}+\frac{\delta}{2}\sqrt{\frac{6m}{A}%
}\left(  -\rho_{1}t_{0}+\frac{16n\kappa\overline{A}}{\sqrt{A\rho_{2}}%
}T\right)
\]
and then
\[
\lbrack2\left\vert \nabla_{b}f\right\vert ^{2}+\beta_{1}t\left\vert
f_{0}\right\vert ^{2}-\delta f_{t}]\left(  x,T\right)  \leq\frac
{C_{1}^{^{\prime}}}{T}+C_{2}^{^{\prime}}%
\]
with $C_{1}^{^{\prime}}:=\frac{m\delta^{2}}{2}$ and
\[
C_{2}^{^{\prime}}:=-\frac{\rho_{1}\delta}{2}\sqrt{\frac{6m}{A}}+\frac
{8n\delta\kappa\overline{A}\sqrt{6m}}{A\sqrt{\rho_{2}}}>0.
\]

$(b)$ If $x\leq\beta t_{0}\left\vert f_{0}\right\vert ^{2}$, it follows that
\[
0\geq-2\beta t_{0}y^{2}+t_{0}\rho_{2}y^{2}-\frac{3}{A}\left(  \rho_{1}%
-\frac{16n\kappa\overline{A}}{\sqrt{A\rho_{2}}}\frac{T}{t_{0}}\right)
^{2}t_{0}%
\]
and then
\[
y^{2}\leq\frac{1}{\left(  \rho_{2}-2\beta\right)  }\frac{3}{A}\left(  \rho
_{1}-\frac{16n\kappa\overline{A}}{\sqrt{A\rho_{2}}}\frac{T}{t_{0}}\right)
^{2}.
\]
Hence%
\[%
\begin{array}
[c]{ccl}%
F & \leq & 2\beta t_{0}^{2}y^{2}\\
& \leq & \frac{6\beta}{\left(  \rho_{2}-2\beta\right)  A}\left(  \rho_{1}%
t_{0}-\frac{16n\kappa\overline{A}}{\sqrt{A\rho_{2}}}T\right)  ^{2}\\
& \leq & \frac{3}{4n\overline{A}\sqrt{A\rho_{2}}}\left(  \rho_{1}\frac{t_{0}%
}{\sqrt{T}}-\frac{16n\kappa\overline{A}}{\sqrt{A\rho_{2}}}\sqrt{T}\right)
^{2}.
\end{array}
\]
Finally we have
\[
\lbrack2\left\vert \nabla_{b}f\right\vert ^{2}+\beta_{1}t\left\vert
f_{0}\right\vert ^{2}-\delta f_{t}]\left(  x,T\right)  \leq C_{2}%
^{^{\prime\prime}}\left(  \rho_{1},\rho_{2},\kappa,m,n,\delta,\overline
{A}\right)
\]
with
\[
C_{2}^{^{\prime\prime}}:=\frac{3}{4n\overline{A}\sqrt{A\rho_{2}}}\left(
-\rho_{1}+\frac{16n\kappa\overline{A}}{\sqrt{A\rho_{2}}}\right)  ^{2}.
\]

$(ii)\ $Under the case%
\[
T\leq T_{0}:=\frac{\sqrt{A\rho_{2}}}{4n\rho_{2}\overline{A}},
\]
we have%
\[
\beta_{1}=\frac{\rho_{2}}{4}%
\]
and then
\[
\mathcal{A\geq}-\frac{3}{A}\left(  \rho_{1}-\frac{4\kappa}{\rho_{2}t_{0}%
}\right)  ^{2}t_{0}.
\]

$(a)\ $If $x\geq\beta t_{0}\left\vert f_{0}\right\vert ^{2}$, \ then%
\[
\lbrack2\left\vert \nabla_{b}f\right\vert ^{2}+\beta_{1}t\left\vert
f_{0}\right\vert ^{2}-\delta f_{t}]\left(  x,T\right)  \leq\frac
{C_{1}^{^{\prime\prime}}}{T}+C_{2}^{^{\prime\prime\prime}}%
\]
with $C_{1}^{^{\prime\prime}}:=\frac{m\delta^{2}}{2}+\frac{2\delta\kappa}%
{\rho_{2}}\sqrt{\frac{6m}{A}}$ and
\[
C_{2}^{^{\prime\prime\prime}}:=-\frac{\delta\rho_{1}}{2}\sqrt{\frac{6m}{A}}>0.
\]

$(b)\ $If $x\leq\beta t_{0}\left\vert f_{0}\right\vert ^{2}$, then%
\[
\lbrack2\left\vert \nabla_{b}f\right\vert ^{2}+\beta_{1}t\left\vert
f_{0}\right\vert ^{2}-\delta f_{t}]\left(  x,T\right)  \leq\frac
{C_{1}^{^{\prime\prime\prime}}}{T}%
\]
with
\[
C_{1}^{^{\prime\prime\prime}}:=\frac{3}{A}\left(  \frac{\rho_{1}\sqrt
{A\rho_{2}}}{4n\rho_{2}\overline{A}}-\frac{4\kappa}{\rho_{2}}\right)  ^{2}.
\]

$(II)$ $Case\ II$ $:\ $ $x\leq0:$

We may assume
\[
\left(  \delta_{0}-2\right)  \left\vert \nabla_{b}f\right\vert ^{2}\leq\beta
t_{0}\left\vert f_{0}\right\vert ^{2}.
\]
Otherwise,
\[
F\leq0.
\]

From $\left(  \ref{16}\right)  $%
\begin{equation}
0\geq-\frac{F}{t_{0}}+2t_{0}\left[  2(\rho_{1}-\frac{\kappa}{\beta t_{0}%
})\frac{\beta t_{0}y^{2}}{\left(  \delta_{0}-2\right)  }+\rho_{2}y^{2}\right]
-16n\beta t_{0}^{2}\overline{A}\overline{x}y-\beta t_{0}y^{2}.\label{2016a}%
\end{equation}
Set
\[
\beta=\beta_{2}:=\min\left\{  \frac{\rho_{2}}{2},\frac{1}{n\overline{A}%
T},\frac{1}{T\left\Vert f_{0}\right\Vert _{M\times\left[  0,T\right]  }%
}\right\}  .
\]
Hence
\[%
\begin{array}
[c]{ccl}%
0 & \geq & -2\beta t_{0}y^{2}+2\rho_{2}t_{0}y^{2}+\rho_{1}\frac{4\beta
t_{0}^{2}}{\left(  \delta_{0}-2\right)  }y^{2}-\frac{4\kappa t_{0}}{\left(
\delta_{0}-2\right)  }y^{2}-16n\beta\overline{A}t_{0}^{2}y\frac{\beta
t_{0}y^{2}}{\left(  \delta_{0}-2\right)  }\\
& = & 2\left(  \rho_{2}-\beta\right)  t_{0}y^{2}+t_{0}y^{2}\left[  \rho
_{1}\frac{4\beta t_{0}}{\left(  \delta_{0}-2\right)  }-\frac{4\kappa}{\left(
\delta_{0}-2\right)  }-\frac{16n\beta^{2}\overline{A}t_{0}^{2}y}{\left(
\delta_{0}-2\right)  }\right] \\
& \geq & t_{0}y^{2}\left[  2\left(  \rho_{2}-\beta\right)  +\rho_{1}\frac
{4}{\left(  \delta_{0}-2\right)  \overline{A}}-\frac{4\kappa}{\left(
\delta_{0}-2\right)  }-\frac{16}{\left(  \delta_{0}-2\right)  }\right] \\
& \geq & t_{0}y^{2}\left[  \rho_{2}+\rho_{1}\frac{4}{\left(  \delta
_{0}-2\right)  \overline{A}}-\frac{4\kappa}{\left(  \delta_{0}-2\right)
}-\frac{16}{\left(  \delta_{0}-2\right)  }\right]  .
\end{array}
\]
Choose $\delta_{0}\left(  \rho_{1},\rho_{2},\kappa,\overline{A}\right)  >2$
such that
\[
\left(  \rho_{2}+\rho_{1}\frac{4}{\left(  \delta_{0}-2\right)  \overline{A}%
}-\frac{4\kappa}{\left(  \delta_{0}-2\right)  }-\frac{16}{\left(  \delta
_{0}-2\right)  }\right)  >0,
\]
we obtain
\[
y(x_{0},t_{0})=0.
\]
It follows that
\[
F(x_{0},t_{0})\leq0
\]
and then
\[
2\left\vert \nabla_{b}f\right\vert ^{2}+\beta_{2}t\left\vert f_{0}\right\vert
^{2}-\delta f_{t}\leq0
\]
on $M\times\lbrack0,T].$ So if we choose
\[
\beta\leq\min\left\{  \beta_{1},\beta_{2},\frac{1}{4\left(  n+1\right)
nT},\frac{1}{2\left(  n+1\right)  \overline{A}T}\right\}
\]

and
\[
m=n+1,\varepsilon_{1}=1,N=\beta T
\]

such that
\[
\frac{2n}{m}-\frac{2n^{2}N}{\varepsilon_{1}}-\frac{2n^{2}N^{2}\overline{A}%
^{2}m}{m-1}>0
\]
with $0<\nu\leq N$ as in \textit{Lemma} \ref{l2}, we obtain
\[
\lbrack\left\vert \nabla_{b}f\right\vert ^{2}-\frac{\delta}{2}f_{t}]\leq
\frac{C_{1}}{t}+C_{2}.
\]
Here%
\[%
\begin{array}
[c]{ccl}%
C_{1} & = & \frac{1}{2}\max\left\{  n\left(  n+1\right)  \delta^{2}%
+\frac{8\sqrt{3}\left(  n+1^{2}\right)  \delta^{2}}{\left(  \delta-\delta
_{0}\right)  },\frac{3n\left(  n+1\right)  \delta^{2}}{4\left(  \delta
-\delta_{0}\right)  ^{2}}\left[  \left(  k+\frac{\overline{B}^{2}}{2\left(
n+1\right)  }\right)  \frac{\left(  \delta-\delta_{0}\right)  }{2n\left(
n+1\right)  \overline{A}\delta}+\frac{16\left(  n+1\right)  }{n}\right]
^{2}\right\}  ,\\
C_{2} & = & \frac{1}{2}\max\{ \left(  k+\frac{\overline{B}^{2}}{2\left(
n+1\right)  }\right)  \frac{\sqrt{3}n\left(  n+1\right)  \delta^{2}}{2\left(
\delta-\delta_{0}\right)  }+16\sqrt{3}\left(  n+1\right)  ^{2}\frac{\delta
^{3}\overline{A}}{\left(  \delta-\delta_{0}\right)  ^{2}},\\
&  & \frac{3\left(  n+1\right)  \delta}{8n\overline{A}\left(  \delta
-\delta_{0}\right)  }\left(  k+\frac{\overline{B}^{2}}{2\left(  n+1\right)
}+\frac{32n\left(  n+1\right)  \delta\overline{A}}{\left(  \delta-\delta
_{0}\right)  }\right)  ^{2}\}.
\end{array}
\]
Note that $\beta\rightarrow0^{+}$ if $T\rightarrow\infty$ and
\[%
\begin{array}
[c]{ccc}%
\max\left\{  C_{1}^{^{\prime}},C_{1}^{^{\prime\prime}},C_{1}^{^{\prime
\prime\prime}}\right\}  & \leq & C_{1},\\
\max\left\{  C_{2}^{^{\prime}},C_{2}^{^{\prime\prime}},C_{2}^{^{\prime
\prime\prime}}\right\}  & \leq & C_{2}.
\end{array}
\]
These will complete the proof.
\end{proof}

\bigskip

The proof of \textbf{Theorem \ref{t2} : }

\begin{proof}
Define%
\[%
\begin{array}
[c]{ccc}%
\eta & : & \left[  t_{1},t_{2}\right]  \longrightarrow M\times\left[
t_{1},t_{2}\right] \\
&  & t\mapsto\left(  \gamma\left(  t\right)  ,t\right)
\end{array}
\]
where $\gamma$ is a horizontal curve with $\gamma\left(  t_{1}\right)  =x_{1}
$, $\gamma\left(  t_{2}\right)  =x_{2}$. Let $f=\ln u$, integrate $f^{\prime
}(t)$ along $\gamma$, so we get%
\[
f\left(  x_{1},t_{1}\right)  -f\left(  x_{2},t_{2}\right)  =-%
{\displaystyle\int\limits_{t_{1}}^{t_{2}}}
\left(  f\circ\eta\right)  ^{\prime}dt=-%
{\displaystyle\int\limits_{t_{1}}^{t_{2}}}
\left(  \left\langle \gamma^{\prime}\left(  t\right)  ,\nabla_{b}%
f\right\rangle +f_{t}\right)  dt.
\]
By applying Theorem \ref{t1}\textbf{, }this yields%

\[%
\begin{array}
[c]{ccl}%
f\left(  x_{1},t_{1}\right)  -f\left(  x_{2},t_{2}\right)  & < & -%
{\displaystyle\int\limits_{t_{1}}^{t_{2}}}
\left\langle \gamma^{\prime}\left(  t\right)  ,\nabla_{b}f\right\rangle dt+%
{\displaystyle\int\limits_{t_{1}}^{t_{2}}}
\frac{1}{\delta}\left(  \frac{C_{1}}{t}+C_{2}-\left\vert \nabla_{b}%
f\right\vert ^{2}\right)  dt\\
& \leq &
{\displaystyle\int\limits_{t_{1}}^{t_{2}}}
\left(  \frac{\delta}{4}\left\vert \gamma^{\prime}\left(  t\right)
\right\vert ^{2}+\frac{C_{1}}{\delta t}+\frac{C_{2}}{\delta}\right)  dt.
\end{array}
\]
We could choose
\[
\left\vert \gamma^{\prime}\left(  t\right)  \right\vert =\frac{ d_{cc}\left(
x_{1},x_{2}\right)  }{t_{2}-t_{1}};
\]
we reach%

\[
\frac{u(x_{1},t_{1})}{u(x_{2},t_{2})}<\left(  \frac{t_{2}}{t_{1}}\right)  ^{
\frac{C_{1}}{\delta}}\cdot\exp\left(  \frac{\delta}{4}\frac{ d_{cc}%
(x_{1},x_{2})^{2}}{t_{2}-t_{1}}+\frac{C_{2}}{\delta} (t_{2}-t_{1})\right)  .
\]

\end{proof}

\bigskip

\section{Li-Yau gradient estimates for sum of squares of vector fields}

In the paper of H.-D. Cao and S.-T. Yau (\cite{cy}), they derived the gradient
estimate for step $2$. Here we generalize the result to higher step under the
assumption of the curvature-dimension inequality. \ Let $M$ be a closed smooth
manifold and $L$ be an operator with respect to the sum of squares of vector
fields $\{e_{1},\ e_{2},\ ...,e_{d}\}$
\[
L=%
{\displaystyle\sum\limits_{j\in I_{d}}}
e_{j}^{2}.
\]
Suppose that $u$ is the positive solution of
\[
(L-\frac{\partial}{\partial t})u=0
\]
on $M\times\lbrack0,+\infty).$ Now we introduce another test function as in
\cite{cy} for $f(x,t)=\ln u(x,t)$
\begin{equation}
G\left(  x,t\right)  =t\left[
{\displaystyle\sum\limits_{j\in I_{d}}}
\left\vert e_{j}f\right\vert ^{2}+%
{\displaystyle\sum\limits_{\alpha\in\Lambda}}
\left(  1+\left\vert Y_{\alpha}f\right\vert ^{2}\right)  ^{\lambda}-\delta
f_{t}\right] \label{51}%
\end{equation}
for $\lambda\in\left(  \frac{1}{2},1\right)  $ to be determined later. Note
that the power $\lambda$ in this test function $G$ is necessary due to
(\ref{2016aaa}).

By the same computation as in \textit{Lemma 2.1} of \cite{cy}, we have Lemma 5.1.

\begin{lemma}
\label{l51} Let $M$ be a smooth connected manifold with a positive measure and
$L$ be an operator with respect to the sum of squares of vector fields
$\{e_{1},\ e_{2},\ ...,e_{d}\}$. Suppose that $u$ is the positive solution of%
\[
(L-\frac{\partial}{\partial t})u=0
\]
on $M\times\lbrack0,+\infty).$ Then$\ $the following equality holds:%
\[%
\begin{array}
[c]{ccl}%
\left(  L-\frac{\partial}{\partial t}\right)  G & = & -\frac{G}{t}+2t\left(
{\displaystyle\sum\limits_{i,j\in I_{d}}}
\left\vert e_{i}e_{j}f\right\vert ^{2}+%
{\displaystyle\sum\limits_{j\in I_{d}}}
\left(  e_{j}f\right)  \left(  \left[  L,e_{j}\right]  f\right)  \right) \\
&  & +2\lambda t%
{\displaystyle\sum\limits_{i\in I_{d},\alpha\in\Lambda}}
\left(  1+\left\vert Y_{\alpha}f\right\vert ^{2}\right)  ^{\lambda
-2}\left\vert e_{i}Y_{\alpha}f\right\vert ^{2}\left[  1+\left(  2\lambda
-1\right)  \left\vert Y_{\alpha}f\right\vert ^{2}\right] \\
&  & +2\lambda t%
{\displaystyle\sum\limits_{\alpha\in\Lambda}}
\left(  1+\left\vert Y_{\alpha}f\right\vert ^{2}\right)  ^{\lambda-1}\left(
Y_{\alpha}f\right)  \left(  \left[  L,Y_{\alpha}\right]  f\right) \\
&  & +4\lambda t%
{\displaystyle\sum\limits_{i\in I_{d},\alpha\in\Lambda}}
\left(  1+\left\vert Y_{\alpha}f\right\vert ^{2}\right)  ^{\lambda-1}\left(
e_{i}f\right)  \left(  Y_{\alpha}f\right)  \left(  \left[  e_{i},Y_{\alpha
}\right]  f\right)  -2%
{\displaystyle\sum\limits_{j\in I_{d}}}
\left(  e_{j}f\right)  \left(  e_{j}G\right)  .
\end{array}
\]

\end{lemma}

Then, as a consequence of Lemma \ref{l51}, we get Proposition 5.2.

\begin{proposition}
If $M$ satisfies the curvature-dimension inequality $CD\left(  \rho_{1}%
,\rho_{2},\kappa,m\right)  $ for $\rho_{1}\in%
\mathbb{R}
,\rho_{2}>0,\kappa\geq0,m>0$, then%
\begin{equation}%
\begin{array}
[c]{ccl}%
\left(  L-\frac{\partial}{\partial t}\right)  G & \geq & -\frac{G}{t}+\frac
{t}{m}\left(  Lf\right)  ^{2}+t%
{\displaystyle\sum\limits_{i,j}}
\left\vert e_{i}e_{j}f\right\vert ^{2}+t\rho_{2}\Gamma^{Z}\left(  f,\ f\right)
\\
&  & +2\lambda\left(  2\lambda-1\right)  t%
{\displaystyle\sum\limits_{i,\alpha}}
\left(  1+\left\vert Y_{\alpha}f\right\vert ^{2}\right)  ^{\lambda
-1}\left\vert e_{i}Y_{\alpha}f\right\vert ^{2}\\
&  & +t%
{\displaystyle\sum\limits_{j}}
\left(  e_{j}f\right)  \left(  \left[  L,e_{j}\right]  f\right)  +t\left(
\rho_{1}-\frac{\kappa}{\nu}\right)  \Gamma\left(  f,\ f\right)  -\nu
t\Gamma_{2}^{Z}\left(  f,\ f\right) \\
&  & +2\lambda t%
{\displaystyle\sum\limits_{\alpha}}
\left(  1+\left\vert Y_{\alpha}f\right\vert ^{2}\right)  ^{\lambda-1}\left(
Y_{\alpha}f\right)  \left(  \left[  L,Y_{\alpha}\right]  f\right) \\
&  & +4\lambda t%
{\displaystyle\sum\limits_{i,\alpha}}
\left(  1+\left\vert Y_{\alpha}f\right\vert ^{2}\right)  ^{\lambda-1}\left(
e_{i}f\right)  \left(  Y_{\alpha}f\right)  \left(  \left[  e_{i},Y_{\alpha
}\right]  f\right)  -2%
{\displaystyle\sum\limits_{j}}
\left(  e_{j}f\right)  \left(  e_{j}G\right)  .
\end{array}
\label{52}%
\end{equation}

\end{proposition}

\begin{remark}
With the help of the curvature-dimension inequality, we obtain the extra
positive term $t\rho_{2}\Gamma^{Z}\left(  f,f\right)  $ in order to control
some of the remaining negative terms in the upcoming estimates. Note that
there are similar spirits as in \cite[(2.10)]{cy} and \cite[(2.10)]{ckl}.
\end{remark}

Now we are ready to prove the main theorem in this section.

The proof of \textbf{Theorem \ref{t3} :}

\begin{proof}
Here we follow the method as in (\cite[Proposition 2.1.]{cy}).\textit{\ }We
claim that there are positive constants $C_{1},C_{2},C_{3}$ such that%
\[
G\leq C_{1}+C_{2}t+C_{3}t^{\frac{2\lambda-1}{\lambda-1}}.
\]
If not, then for arbitrary such $C_{1},C_{2},C_{3},$ we have
\[
G>C_{1}+C_{2}t+C_{3}t^{\frac{2\lambda-1}{\lambda-1}}%
\]
at its maximum $\left(  x_{0},t_{0}\right)  $ on $M\times\left[  0,T\right]  $
for some $T>0$. Clearly,
\[
\left\{
\begin{array}
[c]{lll}%
t_{0}>0 & , & \left(  e_{j}G\right)  \left(  x_{0},t_{0}\right)  =0,\\
\frac{\partial G}{\partial t}\left(  x_{0},t_{0}\right)  \geq0 & , & LG\left(
x_{0},t_{0}\right)  \leq0,
\end{array}
\right.
\]
for $j\in I_{d}.$

Choosing%
\[
\nu=\lambda\left(  2\lambda-1\right)  \left(  1+\max_{\alpha}\left(
\left\vert Y_{\alpha}f\right\vert ^{2}\left(  x_{0},t_{0}\right)  \right)
\right)  ^{\lambda-1}%
\]
and evaluating $\left(  \ref{52}\right)  $ at $\left(  x_{0},t_{0}\right)  $,
we obtain%
\begin{equation}%
\begin{array}
[c]{ccl}%
0 & \geq & -\frac{G}{t_{0}}+\frac{t_{0}}{m}\left(  Lf\right)  ^{2}+t_{0}%
{\displaystyle\sum\limits_{i,j}}
\left\vert e_{i}e_{j}f\right\vert ^{2}+t_{0}\rho_{2}%
{\displaystyle\sum\limits_{\alpha}}
\left\vert Y_{\alpha}f\right\vert ^{2}\\
&  & +t_{0}%
{\displaystyle\sum\limits_{j}}
\left(  e_{j}f\right)  \left(  \left[  L,e_{j}\right]  f\right) \\
&  & +\lambda\left(  2\lambda-1\right)  t_{0}%
{\displaystyle\sum\limits_{j,\alpha}}
\left(  1+\left\vert Y_{\alpha}f\right\vert ^{2}\right)  ^{\lambda
-1}\left\vert e_{j}Y_{\alpha}f\right\vert ^{2}\\
&  & +\lambda\left(  3-2\lambda\right)  t_{0}%
{\displaystyle\sum\limits_{\alpha}}
\left(  1+\left\vert Y_{\alpha}f\right\vert ^{2}\right)  ^{\lambda-1}\left(
Y_{\alpha}f\right)  \left(  \left[  L,Y_{\alpha}\right]  f\right) \\
&  & +4\lambda t_{0}%
{\displaystyle\sum\limits_{j,\alpha}}
\left(  1+\left\vert Y_{\alpha}f\right\vert ^{2}\right)  ^{\lambda-1}\left(
e_{j}f\right)  \left(  Y_{\alpha}f\right)  \left(  \left[  e_{j},Y_{\alpha
}\right]  f\right) \\
&  & -\frac{t_{0}\kappa}{\lambda\left(  2\lambda-1\right)  }%
{\displaystyle\sum\limits_{\alpha}}
\left(  1+\left\vert Y_{\alpha}f\right\vert ^{2}\right)  ^{1-\lambda}%
{\displaystyle\sum\limits_{j}}
\left\vert e_{j}f\right\vert ^{2}+t_{0}\rho_{1}%
{\displaystyle\sum\limits_{j}}
\left\vert e_{j}f\right\vert ^{2}.
\end{array}
\label{53}%
\end{equation}
By straightforward computation, we have%
\begin{equation}%
\begin{array}
[c]{ccl}%
\left\vert \left[  L,e_{j}\right]  f\right\vert  & = & \left\vert 2%
{\displaystyle\sum\limits_{i}}
e_{i}\left[  e_{i},e_{j}\right]  f-%
{\displaystyle\sum\limits_{i}}
\left[  e_{i},\left[  e_{i},e_{j}\right]  \right]  f\right\vert \\
& \leq & 2%
{\displaystyle\sum\limits_{i}}
\left\vert e_{i}\left[  e_{i},e_{j}\right]  f\right\vert +%
{\displaystyle\sum\limits_{\alpha}}
\left\vert Y_{\alpha}f\right\vert
\end{array}
\label{54}%
\end{equation}
and%

\begin{equation}%
\begin{array}
[c]{ccl}%
\left\vert \left[  L,Y_{\eta}^{\prime}\right]  f\right\vert  & = & \left\vert
2%
{\displaystyle\sum\limits_{i}}
e_{i}\left[  e_{i},Y_{\eta}^{\prime}\right]  f-%
{\displaystyle\sum\limits_{i}}
\left[  e_{i},\left[  e_{i},Y_{\eta}^{\prime}\right]  \right]  f\right\vert \\
& \leq & 2%
{\displaystyle\sum\limits_{i,A}}
\left\vert e_{i}Y_{A}^{\prime\prime}f\right\vert +d\left(  a%
{\displaystyle\sum\limits_{j}}
\left\vert e_{j}f\right\vert +b%
{\displaystyle\sum\limits_{\eta}}
\left\vert Y_{\eta}^{\prime}f\right\vert +c%
{\displaystyle\sum\limits_{A}}
\left\vert Y_{A}^{\prime\prime}f\right\vert \right) \\
& \leq & 2%
{\displaystyle\sum\limits_{i,\alpha}}
\left\vert e_{i}Y_{\alpha}f\right\vert +da%
{\displaystyle\sum\limits_{j}}
\left\vert e_{j}f\right\vert +d\left(  b+c\right)
{\displaystyle\sum\limits_{\alpha}}
\left\vert Y_{\alpha}f\right\vert .
\end{array}
\label{55}%
\end{equation}
Similarly%

\begin{equation}%
\begin{array}
[c]{ccl}%
\left\vert \left[  L,Y_{A}^{\prime\prime}\right]  f\right\vert  & = &
\left\vert
{\displaystyle\sum\limits_{i}}
e_{i}\left[  e_{i},Y_{A}^{\prime\prime}\right]  f+%
{\displaystyle\sum\limits_{i}}
\left[  e_{i},Y_{A}^{\prime\prime}\right]  e_{i}f\right\vert \\
& \leq &
{\displaystyle\sum\limits_{i}}
\left\vert e_{i}\left(  a_{\left(  i,A\right)  }^{n}e_{n}+b_{\left(
i,A\right)  }^{\eta}Y_{\eta}^{\prime}+c_{\left(  i,A\right)  }^{B}%
Y_{B}^{\prime\prime}\right)  f\right\vert \\
&  & +a%
{\displaystyle\sum\limits_{i,n}}
\left\vert e_{n}e_{i}f\right\vert +b%
{\displaystyle\sum\limits_{i,\eta}}
\left\vert Y_{\eta}^{\prime}e_{i}f\right\vert +c%
{\displaystyle\sum\limits_{i,B}}
\left\vert Y_{B}^{\prime\prime}e_{i}f\right\vert \\
& \leq & 2a%
{\displaystyle\sum\limits_{i,j}}
\left\vert e_{i}e_{j}f\right\vert +2b%
{\displaystyle\sum\limits_{i,\eta}}
\left\vert e_{i}Y_{\eta}^{\prime}f\right\vert \\
&  & +2c%
{\displaystyle\sum\limits_{i,A}}
\left\vert e_{i}Y_{A}^{\prime\prime}f\right\vert +b%
{\displaystyle\sum\limits_{i,\eta}}
\left\vert \left[  e_{i},Y_{\eta}^{\prime}\right]  f\right\vert \\
&  & +c%
{\displaystyle\sum\limits_{i,A}}
\left\vert \left[  e_{i},Y_{A}^{\prime\prime}\right]  f\right\vert
+da^{\prime}%
{\displaystyle\sum\limits_{j}}
\left\vert e_{j}f\right\vert \\
&  & +db^{\prime}%
{\displaystyle\sum\limits_{\eta}}
\left\vert Y_{\eta}^{\prime}f\right\vert +dc^{\prime}%
{\displaystyle\sum\limits_{B}}
\left\vert Y_{B}^{\prime\prime}f\right\vert \\
& \leq & 2a%
{\displaystyle\sum\limits_{i,j}}
\left\vert e_{i}e_{j}f\right\vert +2\left(  b+c\right)
{\displaystyle\sum\limits_{i,\alpha}}
\left\vert e_{i}Y_{\alpha}f\right\vert +da^{\prime}%
{\displaystyle\sum\limits_{j}}
\left\vert e_{j}f\right\vert +db^{\prime}%
{\displaystyle\sum\limits_{\eta}}
\left\vert Y_{\eta}^{\prime}f\right\vert \\
&  & +\left(  dc^{\prime}+b\right)
{\displaystyle\sum\limits_{B}}
\left\vert Y_{B}^{\prime\prime}f\right\vert +c%
{\displaystyle\sum\limits_{i,B}}
\left\vert \left(  a_{\left(  i,B\right)  }^{n}e_{n}+b_{\left(  i,B\right)
}^{\eta}Y_{\eta}^{\prime}+c_{\left(  i,B\right)  }^{A}Y_{A}^{\prime\prime
}\right)  f\right\vert \\
& \leq & 2a%
{\displaystyle\sum\limits_{i,j}}
\left\vert e_{i}e_{j}f\right\vert +2\left(  b+c\right)
{\displaystyle\sum\limits_{i,\alpha}}
\left\vert e_{i}Y_{\alpha}f\right\vert +d\left(  a^{\prime}+acd^{4}\right)
{\displaystyle\sum\limits_{j}}
\left\vert e_{j}f\right\vert \\
&  & +d\left(  b^{\prime}+bcd^{4}\right)
{\displaystyle\sum\limits_{\eta}}
\left\vert Y_{\eta}^{\prime}f\right\vert +\left(  dc^{\prime}+b+c^{2}%
d^{4}\right)
{\displaystyle\sum\limits_{B}}
\left\vert Y_{B}^{\prime\prime}f\right\vert \\
& \leq & 2a%
{\displaystyle\sum\limits_{i,j}}
\left\vert e_{i}e_{j}f\right\vert +2\left(  b+c\right)
{\displaystyle\sum\limits_{i,\alpha}}
\left\vert e_{i}Y_{\alpha}f\right\vert +d\left(  a^{\prime}+acd^{4}\right)
{\displaystyle\sum\limits_{j}}
\left\vert e_{j}f\right\vert \\
&  & +\left(  db^{\prime}+bcd^{5}+dc^{\prime}+b+c^{2}d^{4}\right)
{\displaystyle\sum\limits_{\alpha}}
\left\vert Y_{\alpha}f\right\vert .
\end{array}
\label{56}%
\end{equation}
Also%

\begin{equation}
\left\vert \left[  e_{i},Y_{\eta}^{\prime}\right]  f\right\vert \leq%
{\displaystyle\sum\limits_{\alpha}}
\left\vert Y_{\alpha}f\right\vert \label{57}%
\end{equation}
and
\begin{equation}
\left\vert \left[  e_{i},Y_{A}^{\prime\prime}\right]  f\right\vert \leq a%
{\displaystyle\sum\limits_{j}}
\left\vert e_{j}f\right\vert +\left(  b+c\right)
{\displaystyle\sum\limits_{\alpha}}
\left\vert Y_{\alpha}f\right\vert .\label{58}%
\end{equation}
Substituting $\left(  \ref{54}\right)  -\left(  \ref{58}\right)  $ into
$\left(  \ref{53}\right)  $ and noting that%
\[
\left(  L-\frac{\partial}{\partial t}\right)  f=-%
{\displaystyle\sum\limits_{j}}
\left\vert e_{j}f\right\vert ^{2},
\]
we have%
\begin{equation}%
\begin{array}
[c]{ccl}%
0 & \geq & -\frac{G}{t_{0}}+\frac{t_{0}}{m}\left(
{\displaystyle\sum\limits_{j}}
\left\vert e_{j}f\right\vert ^{2}-f_{t}\right)  ^{2}+t_{0}%
{\displaystyle\sum\limits_{i,j}}
\left\vert e_{i}e_{j}f\right\vert ^{2}+t_{0}\rho_{2}%
{\displaystyle\sum\limits_{\alpha}}
\left\vert Y_{\alpha}f\right\vert ^{2}+\\
&  & +\lambda\left(  2\lambda-1\right)  t_{0}%
{\displaystyle\sum\limits_{j,\alpha}}
\left(  1+\left\vert Y_{\alpha}f\right\vert ^{2}\right)  ^{\lambda
-1}\left\vert e_{j}Y_{\alpha}f\right\vert ^{2}-2t_{0}\underset{\left(
1\right)  }{\underbrace{%
{\displaystyle\sum\limits_{i,j}}
\left\vert e_{j}f\right\vert \left\vert e_{i}\left[  e_{i},e_{j}\right]
f\right\vert }}\\
&  & -t_{0}\underset{\left(  2\right)  }{\underbrace{\left(
{\displaystyle\sum\limits_{j}}
\left(  e_{j}f\right)  \right)  \left(
{\displaystyle\sum\limits_{\alpha}}
\left\vert Y_{\alpha}f\right\vert \right)  }}\\
&  & -\underset{\left(  3\right)  }{\underbrace{2a\lambda\left(
3-2\lambda\right)  t_{0}%
{\displaystyle\sum\limits_{i,j,\beta}}
\left(  1+\left\vert Y_{\beta}f\right\vert ^{2}\right)  ^{\lambda-1}\left\vert
Y_{\beta}f\right\vert \left\vert e_{i}e_{j}f\right\vert }}\\
&  & -\underset{\left(  4\right)  }{\underbrace{2\left(  1+b+c\right)
\lambda\left(  3-2\lambda\right)  t_{0}%
{\displaystyle\sum\limits_{i,\alpha,\beta}}
\left(  1+\left\vert Y_{\beta}f\right\vert ^{2}\right)  ^{\lambda-1}\left\vert
Y_{\beta}f\right\vert \left\vert e_{i}Y_{\alpha}f\right\vert }}\\
&  & -d\left(  a+a^{\prime}+acd^{4}\right)  \lambda\left(  3-2\lambda\right)
t_{0}%
{\displaystyle\sum\limits_{j,\beta}}
\left(  1+\left\vert Y_{\beta}f\right\vert ^{2}\right)  ^{\lambda-1}\left\vert
Y_{\beta}f\right\vert \left\vert e_{j}f\right\vert \\
&  & -\left(  db+dc+db^{\prime}+bcd^{5}+dc^{\prime}+b+c^{2}d^{4}\right)
\lambda\left(  3-2\lambda\right)  t_{0}%
{\displaystyle\sum\limits_{\alpha,\beta}}
\left(  1+\left\vert Y_{\alpha}f\right\vert ^{2}\right)  ^{\lambda
-1}\left\vert Y_{\alpha}f\right\vert \left\vert Y_{\beta}f\right\vert \\
&  & -\underset{\left(  5\right)  }{\underbrace{4a\lambda t_{0}\left(
{\displaystyle\sum\limits_{\beta}}
\left(  1+\left\vert Y_{\beta}f\right\vert ^{2}\right)  ^{\lambda-1}\left\vert
Y_{\beta}f\right\vert \right)  \left(
{\displaystyle\sum\limits_{j}}
\left\vert e_{j}f\right\vert \right)  ^{2}}}\\
&  & -\underset{\left(  6\right)  }{\underbrace{4\left(  1+b+c\right)  \lambda
t_{0}%
{\displaystyle\sum\limits_{j,\alpha,\beta}}
\left(  1+\left\vert Y_{\beta}f\right\vert ^{2}\right)  ^{\lambda-1}\left\vert
e_{j}f\right\vert \left\vert Y_{\alpha}f\right\vert \left\vert Y_{\beta
}f\right\vert }}\\
&  & -\underset{\left(  7\right)  }{\underbrace{\frac{t_{0}\kappa}%
{\lambda\left(  2\lambda-1\right)  }%
{\displaystyle\sum\limits_{\alpha}}
\left(  1+\left\vert Y_{\alpha}f\right\vert ^{2}\right)  ^{1-\lambda}\cdot%
{\displaystyle\sum\limits_{j}}
\left\vert e_{j}f\right\vert ^{2}}}-t_{0}\rho_{1}%
{\displaystyle\sum\limits_{j}}
\left\vert e_{j}f\right\vert ^{2}.
\end{array}
\label{59}%
\end{equation}

Now we estimate each term $\left(  1\right)  -\left(  7\right)  $ in the right
hand of $\left(  \ref{59}\right)  $ as follows :%

\[%
\begin{array}
[c]{ccl}%
\left(  1\right)  \text{ }%
{\displaystyle\sum\limits_{i,j}}
\left\vert e_{j}f\right\vert \left\vert e_{i}\left[  e_{i},e_{j}\right]
f\right\vert  & \leq &
{\displaystyle\sum\limits_{i,j,\alpha}}
\left\vert e_{j}f\right\vert \left\vert e_{i}Y_{\alpha}f\right\vert \\
& \leq & \frac{d^{2}}{\lambda\left(  2\lambda-1\right)  }\left(
{\displaystyle\sum\limits_{j}}
\left\vert e_{j}f\right\vert ^{2}\right)  \left(
{\displaystyle\sum\limits_{\alpha}}
\left(  1+\left\vert Y_{\alpha}f\right\vert ^{2}\right)  ^{\frac{1-\lambda}%
{2}}\right)  ^{2}\\
&  & +\frac{\lambda\left(  2\lambda-1\right)  }{4}%
{\displaystyle\sum\limits_{i,\alpha}}
\left(  1+\left\vert Y_{\alpha}f\right\vert ^{2}\right)  ^{\lambda
-1}\left\vert e_{i}Y_{\alpha}f\right\vert ^{2}\\
& \leq & \varepsilon\left(
{\displaystyle\sum\limits_{j}}
\left\vert e_{j}f\right\vert ^{2}\right)  ^{2}+\frac{d^{4}}{\varepsilon\left(
2\lambda-1\right)  ^{2}}\left(
{\displaystyle\sum\limits_{\alpha}}
\left(  1+\left\vert Y_{\alpha}f\right\vert ^{2}\right)  ^{\frac{1-\lambda}%
{2}}\right)  ^{4}\\
&  & +\frac{\lambda\left(  2\lambda-1\right)  }{4}%
{\displaystyle\sum\limits_{i,\alpha}}
\left(  1+\left\vert Y_{\alpha}f\right\vert ^{2}\right)  ^{\lambda
-1}\left\vert e_{i}Y_{\alpha}f\right\vert ^{2}.
\end{array}
\]

\[%
\begin{array}
[c]{ccl}%
\left(  2\right)  \text{ }\left(
{\displaystyle\sum\limits_{j}}
\left(  e_{j}f\right)  \right)  \left(
{\displaystyle\sum\limits_{\alpha}}
\left\vert Y_{\alpha}f\right\vert \right)  & \leq & \frac{4d^{2}\left(
1+d\right)  }{\rho_{2}}\left(
{\displaystyle\sum\limits_{j}}
\left\vert e_{j}f\right\vert \right)  ^{2}+\frac{\rho_{2}}{16d^{2}\left(
1+d\right)  }\left(
{\displaystyle\sum\limits_{\alpha}}
\left\vert Y_{\alpha}f\right\vert \right)  ^{2}\\
& \leq & \frac{4d^{2}\left(  1+d\right)  }{\rho_{2}}\left(
{\displaystyle\sum\limits_{j}}
\left\vert e_{j}f\right\vert \right)  ^{2}+\frac{\rho_{2}}{16}\left(
{\displaystyle\sum\limits_{\alpha}}
\left\vert Y_{\alpha}f\right\vert ^{2}\right)  .
\end{array}
\]

\[%
\begin{array}
[c]{ccl}%
\left(  3\right)  &  & t_{0}\left[  2a\lambda\left(  3-2\lambda\right)
{\displaystyle\sum\limits_{i,j,\beta}}
\left(  1+\left\vert Y_{\beta}f\right\vert ^{2}\right)  ^{\lambda-1}\left\vert
Y_{\beta}f\right\vert \left\vert e_{i}e_{j}f\right\vert \right] \\
& = & t_{0}\left[  \left(  2a\lambda\left(  3-2\lambda\right)
{\displaystyle\sum\limits_{\beta}}
\left(  1+\left\vert Y_{\beta}f\right\vert ^{2}\right)  ^{\lambda-1}\left\vert
Y_{\beta}f\right\vert \right)  \left(
{\displaystyle\sum\limits_{i,j}}
\left\vert e_{i}e_{j}f\right\vert \right)  \right] \\
& \leq & t_{0}\left[
{\displaystyle\sum\limits_{i,j}}
\left\vert e_{i}e_{j}f\right\vert ^{2}+d^{2}\lambda^{2}\left(  3-2\lambda
\right)  ^{2}a^{2}\left(
{\displaystyle\sum\limits_{\beta}}
\left(  1+\left\vert Y_{\beta}f\right\vert ^{2}\right)  ^{\lambda-1}\left\vert
Y_{\beta}f\right\vert \right)  ^{2}\right]  .
\end{array}
\]

\[%
\begin{array}
[c]{ccl}%
\left(  4\right)  \text{ } &  & 2\left(  1+b+c\right)  \lambda\left(
3-2\lambda\right)  t_{0}%
{\displaystyle\sum\limits_{i,\alpha,\beta}}
\left(  1+\left\vert Y_{\beta}f\right\vert ^{2}\right)  ^{\lambda-1}\left\vert
Y_{\beta}f\right\vert \left\vert e_{i}Y_{\alpha}f\right\vert \\
& = & t_{0}\gamma%
{\displaystyle\sum\limits_{i,\alpha,\beta}}
\left[  \left(  1+\left\vert Y_{\beta}f\right\vert ^{2}\right)  ^{\lambda
-1}\left\vert Y_{\beta}f\right\vert \left(  1+\left\vert Y_{\alpha
}f\right\vert ^{2}\right)  ^{\frac{1-\lambda}{2}}\left(  1+\left\vert
Y_{\alpha}f\right\vert ^{2}\right)  ^{\frac{\lambda-1}{2}}\left\vert
e_{i}Y_{\alpha}f\right\vert \right] \\
& \leq & t_{0}\gamma%
{\displaystyle\sum\limits_{i,\alpha,\beta}}
[\frac{d^{2}\left(  1+d\right)  \gamma}{2\lambda\left(  2\lambda-1\right)
}\left(  \left(  1+\left\vert Y_{\beta}f\right\vert ^{2}\right)  ^{\lambda
-1}\left\vert Y_{\beta}f\right\vert \right)  ^{2}\left(  1+\left\vert
Y_{\alpha}f\right\vert ^{2}\right)  ^{1-\lambda}\\
&  & +\frac{\lambda\left(  2\lambda-1\right)  }{2\gamma d^{2}\left(
1+d\right)  }\left(  1+\left\vert Y_{\alpha}f\right\vert ^{2}\right)
^{\lambda-1}\left\vert e_{i}Y_{\alpha}f\right\vert ^{2}]\\
& \leq & \frac{\gamma^{2}d^{3}\left(  1+d\right)  }{2\lambda\left(
2\lambda-1\right)  }t_{0}\left(
{\displaystyle\sum\limits_{\beta}}
\left(  1+\left\vert Y_{\beta}f\right\vert ^{2}\right)  ^{\lambda-1}\left\vert
Y_{\beta}f\right\vert \right)  ^{2}\left(
{\displaystyle\sum\limits_{\alpha}}
\left(  1+\left\vert Y_{\alpha}f\right\vert ^{2}\right)  ^{\frac{1-\lambda}%
{2}}\right)  ^{2}\\
&  & +\frac{\lambda\left(  2\lambda-1\right)  }{2}t_{0}%
{\displaystyle\sum\limits_{i,\alpha}}
\left(  1+\left\vert Y_{\alpha}f\right\vert ^{2}\right)  ^{\lambda
-1}\left\vert e_{i}Y_{\alpha}f\right\vert ^{2}\ for\text{ }\gamma=2\left(
1+b+c\right)  \lambda\left(  3-2\lambda\right)  .
\end{array}
\]

\[%
\begin{array}
[c]{ccl}%
\left(  5\right)  \text{ } &  & 4a\lambda t_{0}\left(
{\displaystyle\sum\limits_{\beta}}
\left(  1+\left\vert Y_{\beta}f\right\vert ^{2}\right)  ^{\lambda-1}\left\vert
Y_{\beta}f\right\vert \right)  \left(
{\displaystyle\sum\limits_{j}}
\left\vert e_{j}f\right\vert \right)  ^{2}\\
& \leq & \varepsilon\lambda t_{0}\left(
{\displaystyle\sum\limits_{j}}
\left\vert e_{j}f\right\vert ^{2}\right)  ^{2}+\frac{4\lambda t_{0}d^{2}a^{2}%
}{\varepsilon}\left(
{\displaystyle\sum\limits_{\beta}}
\left(  1+\left\vert Y_{\beta}f\right\vert ^{2}\right)  ^{\lambda-1}\left\vert
Y_{\beta}f\right\vert \right)  ^{2}.
\end{array}
\]

\[%
\begin{array}
[c]{ccl}%
\left(  6\right)  \text{ } &  & 4\left(  1+b+c\right)  \lambda t_{0}%
{\displaystyle\sum\limits_{j,\alpha,\beta}}
\left(  1+\left\vert Y_{\beta}f\right\vert ^{2}\right)  ^{\lambda-1}\left\vert
e_{j}f\right\vert \left\vert Y_{\alpha}f\right\vert \left\vert Y_{\beta
}f\right\vert \\
& \leq & t_{0}\frac{64d^{2}\left(  1+d\right)  \left(  1+b+c\right)  ^{2}%
}{\rho_{2}}\left(
{\displaystyle\sum\limits_{\beta}}
\left(  1+\left\vert Y_{\beta}f\right\vert ^{2}\right)  ^{\lambda-1}\left\vert
Y_{\beta}f\right\vert \right)  ^{2}\left(
{\displaystyle\sum\limits_{j}}
\left\vert e_{j}f\right\vert \right)  ^{2}\\
&  & +t_{0}\frac{\rho_{2}}{16d^{2}\left(  1+d\right)  }\left(
{\displaystyle\sum\limits_{\alpha}}
\left\vert Y_{\beta}f\right\vert \right)  ^{2}\\
& \leq & t_{0}\frac{\varepsilon}{d^{2}}\left(
{\displaystyle\sum\limits_{j}}
\left\vert e_{j}f\right\vert \right)  ^{4}+t_{0}\frac{\rho_{2}}{16}\left(
{\displaystyle\sum\limits_{\alpha}}
\left\vert Y_{\alpha}f\right\vert ^{2}\right) \\
&  & +t_{0}\frac{1024d^{6}\left(  1+d\right)  ^{2}\left(  1+b+c\right)  ^{4}%
}{\varepsilon\rho_{2}^{2}}\left(
{\displaystyle\sum\limits_{\beta}}
\left(  1+\left\vert Y_{\beta}f\right\vert ^{2}\right)  ^{\lambda-1}\left\vert
Y_{\beta}f\right\vert \right)  ^{4}\\
& \leq & t_{0}\varepsilon\left(
{\displaystyle\sum\limits_{j}}
\left\vert e_{j}f\right\vert ^{2}\right)  ^{2}+t_{0}\frac{\rho_{2}}{16}\left(
%
{\displaystyle\sum\limits_{\alpha}}
\left\vert Y_{\beta}f\right\vert ^{2}\right) \\
&  & +t_{0}\frac{1024d^{6}\left(  1+d\right)  ^{2}\left(  1+b+c\right)  ^{4}%
}{\varepsilon\rho_{2}^{2}}\left(
{\displaystyle\sum\limits_{\beta}}
\left(  1+\left\vert Y_{\beta}f\right\vert ^{2}\right)  ^{\lambda-1}\left\vert
Y_{\beta}f\right\vert \right)  ^{4}.
\end{array}
\]

\[%
\begin{array}
[c]{ccl}%
\left(  7\right)  \text{ } &  & \frac{t_{0}\kappa}{\lambda\left(
2\lambda-1\right)  }%
{\displaystyle\sum\limits_{\alpha}}
\left(  1+\left\vert Y_{\alpha}f\right\vert ^{2}\right)  ^{1-\lambda}\cdot%
{\displaystyle\sum\limits_{j}}
\left\vert e_{j}f\right\vert ^{2}\\
& \leq & t_{0}\left[  \varepsilon\left(
{\displaystyle\sum\limits_{j}}
\left\vert e_{j}f\right\vert ^{2}\right)  ^{2}+\frac{1}{4\varepsilon}%
\frac{\kappa^{2}}{\lambda^{2}\left(  2\lambda-1\right)  ^{2}}\left(
{\displaystyle\sum\limits_{\alpha}}
\left(  1+\left\vert Y_{\alpha}f\right\vert ^{2}\right)  ^{\frac{1-\lambda}%
{2}}\right)  ^{4}\right]  .
\end{array}
\]

Let%
\[%
\begin{array}
[c]{ccl}%
\overline{x} & = &
{\displaystyle\sum\limits_{j}}
\left\vert e_{j}f\right\vert ^{2}\left(  x_{0},t_{0}\right)  ,\\
x & = & \left(  \delta_{0}%
{\displaystyle\sum\limits_{j}}
\left\vert e_{j}f\right\vert ^{2}-\delta f_{t}\right)  \left(  x_{0}%
,t_{0}\right)  ,\\
y & = & \max_{\alpha}\left\vert Y_{\alpha}f\right\vert \left(  x_{0}%
,t_{0}\right)  .
\end{array}
\]
We may assume%
\[
y>1;
\]
otherwise, the similar method adopted as follows still holds for $y\leq1$.

Now we divide it into two cases:

$(I)\ Case\ I$ $:$ $x\geq0:$

In this case, we have%
\begin{equation}
\left(
{\displaystyle\sum\limits_{j}}
\left\vert e_{j}f\right\vert ^{2}-f_{t}\right)  ^{2}\geq\frac{x^{2}}%
{\delta^{2}}+\frac{\left(  \delta-\delta_{0}\right)  ^{2}}{\delta^{2}}\left(
{\displaystyle\sum\limits_{j}}
\left\vert e_{j}f\right\vert ^{2}\right)  ^{2}.\label{60}%
\end{equation}
Substituting $\left(  1\right)  -\left(  7\right)  $ and $\left(
\ref{60}\right)  $ into $\left(  \ref{59}\right)  $, we obtain%
\begin{equation}%
\begin{array}
[c]{ccl}%
0 & \geq & -\frac{G}{t_{0}}+\left\{  \frac{t_{0}}{m\delta^{2}}x^{2}%
+\frac{\left(  \delta-\delta_{0}\right)  ^{2}}{m\delta^{2}}t_{0}\left(
{\displaystyle\sum\limits_{j}}
\left\vert e_{j}f\right\vert ^{2}\right)  ^{2}+\frac{7}{8}\rho_{2}t_{0}\left(
%
{\displaystyle\sum\limits_{\alpha}}
\left\vert Y_{\alpha}f\right\vert ^{2}\right)  \right\} \\
&  & -2t_{0}\varepsilon\left(
{\displaystyle\sum\limits_{j}}
\left\vert e_{j}f\right\vert ^{2}\right)  ^{2}-\frac{2d^{4}}{\varepsilon
\left(  2\lambda-1\right)  ^{2}}t_{0}\left(
{\displaystyle\sum\limits_{\alpha}}
\left(  1+\left\vert Y_{\alpha}f\right\vert ^{2}\right)  ^{\frac{1-\lambda}%
{2}}\right)  ^{4}\\
&  & -\frac{4d^{2}\left(  1+d\right)  }{\rho_{2}}t_{0}\left(
{\displaystyle\sum\limits_{j}}
\left\vert e_{j}f\right\vert \right)  ^{2}-t_{0}d^{2}\lambda^{2}\left(
3-2\lambda\right)  ^{2}a^{2}\left(
{\displaystyle\sum\limits_{\alpha}}
\left(  1+\left\vert Y_{\alpha}f\right\vert ^{2}\right)  ^{\lambda
-1}\left\vert Y_{\alpha}f\right\vert \right)  ^{2}-\\
&  & -\frac{\gamma^{2}d^{3}\left(  1+d\right)  }{2\lambda\left(
2\lambda-1\right)  }t_{0}\left(
{\displaystyle\sum\limits_{\beta}}
\left(  1+\left\vert Y_{\beta}f\right\vert ^{2}\right)  ^{\lambda-1}\left\vert
Y_{\beta}f\right\vert \right)  ^{2}\left(
{\displaystyle\sum\limits_{\alpha}}
\left(  1+\left\vert Y_{\alpha}f\right\vert ^{2}\right)  ^{\frac{1-\lambda}%
{2}}\right)  ^{2}\\
&  & -d\left(  a+a^{\prime}+acd^{4}\right)  \lambda\left(  3-2\lambda\right)
t_{0}%
{\displaystyle\sum\limits_{j,\beta}}
\left(  1+\left\vert Y_{\beta}f\right\vert ^{2}\right)  ^{\lambda-1}\left\vert
e_{j}f\right\vert \left\vert Y_{\beta}f\right\vert \\
&  & -\left(  db+dc+db^{\prime}+bcd^{5}+dc^{\prime}+b+c^{2}d^{4}\right)
\lambda\left(  3-2\lambda\right)  t_{0}%
{\displaystyle\sum\limits_{\alpha,\beta}}
\left(  1+\left\vert Y_{\alpha}f\right\vert ^{2}\right)  ^{\lambda
-1}\left\vert Y_{\alpha}f\right\vert \left\vert Y_{\beta}f\right\vert \\
&  & -t_{0}\varepsilon\lambda\left(
{\displaystyle\sum\limits_{j}}
\left\vert e_{j}f\right\vert ^{2}\right)  ^{2}-\frac{4\lambda t_{0}d^{2}a^{2}%
}{\varepsilon}\left(
{\displaystyle\sum\limits_{\beta}}
\left(  1+\left\vert Y_{\beta}f\right\vert ^{2}\right)  ^{\lambda-1}\left\vert
Y_{\beta}f\right\vert \right)  ^{2}-t_{0}\varepsilon\left(
{\displaystyle\sum\limits_{j}}
\left\vert e_{j}f\right\vert ^{2}\right)  ^{2}\\
&  & -t_{0}\frac{1024d^{6}\left(  1+d\right)  ^{2}\left(  1+b+c\right)  ^{4}%
}{\varepsilon\rho_{2}^{2}}\left(
{\displaystyle\sum\limits_{\beta}}
\left(  1+\left\vert Y_{\beta}f\right\vert ^{2}\right)  ^{\lambda-1}\left\vert
Y_{\beta}f\right\vert \right)  ^{4}-t_{0}\varepsilon\left(
{\displaystyle\sum\limits_{j}}
\left\vert e_{j}f\right\vert ^{2}\right)  ^{2}\\
&  & -t_{0}\frac{1}{4\varepsilon}\frac{\kappa^{2}}{\lambda^{2}\left(
2\lambda-1\right)  ^{2}}\left(
{\displaystyle\sum\limits_{\alpha}}
\left(  1+\left\vert Y_{\alpha}f\right\vert ^{2}\right)  ^{\frac{1-\lambda}%
{2}}\right)  ^{4}-t_{0}\rho_{1}\left(
{\displaystyle\sum\limits_{j}}
\left\vert e_{j}f\right\vert ^{2}\right)  .
\end{array}
\label{61}%
\end{equation}
Because%
\[%
{\displaystyle\sum\limits_{j,\beta}}
\left(  1+\left\vert Y_{\beta}f\right\vert ^{2}\right)  ^{\lambda-1}\left\vert
e_{j}f\right\vert \left\vert Y_{\beta}f\right\vert \leq\frac{1}{4}\left(
{\displaystyle\sum\limits_{j}}
\left\vert e_{j}f\right\vert \right)  ^{2}+\left(
{\displaystyle\sum\limits_{\alpha}}
\left(  1+\left\vert Y_{\alpha}f\right\vert ^{2}\right)  ^{\lambda
-1}\left\vert Y_{\alpha}f\right\vert \right)  ^{2},
\]
we could write the inequality $\left(  \ref{61}\right)  $ in $\overline{x},y$
:%
\[%
\begin{array}
[c]{ccl}%
0 & \geq & -\frac{G}{t_{0}}+\frac{t_{0}}{m\delta^{2}}x^{2}+\frac{3}{4}\rho
_{2}t_{0}y^{2}+t_{0}[\frac{\left(  \delta-\delta_{0}\right)  ^{2}}{m\delta
^{2}}\overline{x}^{2}-2\varepsilon\overline{x}^{2}-\frac{4d^{3}\left(
1+d\right)  }{\rho_{2}}\overline{x}\\
&  & -\frac{d^{2}\left(  a+a^{\prime}+acd^{4}\right)  }{4}\left(
3-2\lambda\right)  \lambda\overline{x}-\lambda\varepsilon\overline{x}%
^{2}-\varepsilon\overline{x}^{2}-\varepsilon\overline{x}^{2}-\rho_{1}%
\overline{x}]\\
&  & +t_{0}[\frac{1}{8}\rho_{2}y^{2}-\frac{2d^{4}}{\varepsilon\left(
2\lambda-1\right)  ^{2}}y^{4\left(  1-\lambda\right)  }-d^{2}\lambda
^{2}\left(  3-2\lambda\right)  ^{2}a^{2}y^{2\left(  2\lambda-1\right)  }\\
&  & -\frac{\gamma^{2}d^{3}\left(  1+d\right)  }{2\lambda\left(
2\lambda-1\right)  }y^{2\lambda}-d\left(  a+a^{\prime}+acd^{4}\right)
\lambda\left(  3-2\lambda\right)  y^{2\left(  2\lambda-1\right)  }\\
&  & -\left(  db+dc+db^{\prime}+bcd^{4}+dc^{\prime}+b+c^{2}d^{4}\right)
\lambda\left(  3-2\lambda\right)  y^{2\lambda}-\\
&  & -\frac{4\lambda d^{2}a^{2}}{\varepsilon}y^{2\left(  2\lambda-1\right)
}-\frac{1024d^{6}\left(  1+d\right)  ^{2}\left(  1+b+c\right)  ^{4}%
}{\varepsilon\rho_{2}^{2}}y^{4\left(  2\lambda-1\right)  }-\frac
{1}{4\varepsilon}\frac{\kappa^{2}}{\lambda^{2}\left(  2\lambda-1\right)  ^{2}%
}y^{4\left(  1-\lambda\right)  }\\
&  & -\mathrm{lower}\text{ }\mathrm{order}\text{ }\mathrm{terms}].
\end{array}
\]
Choose%
\[
\varepsilon=\frac{\left(  \delta-\delta_{0}\right)  ^{2}}{10m\delta^{2}},
\]
we obtain%
\begin{equation}
0\geq-\frac{G}{t_{0}}+\frac{t_{0}}{m\delta^{2}}x^{2}+\frac{3}{4}\rho_{2}%
t_{0}y^{2}-C_{4}t_{0}.\label{64}%
\end{equation}

$(i)\ $If $x\geq%
{\displaystyle\sum\limits_{\alpha}}
\left(  1+\left\vert Y_{\alpha}f\right\vert ^{2}\right)  ^{\lambda}$, then by
the definition
\[
G\left(  x_{0},t_{0}\right)  =t_{0}\left[  \left(  1-\delta_{0}\right)
{\displaystyle\sum\limits_{j}}
\left\vert e_{j}f\right\vert ^{2}+x+%
{\displaystyle\sum\limits_{\alpha}}
\left(  1+\left\vert Y_{\alpha}f\right\vert ^{2}\right)  ^{\lambda}\right]
\]
we have%
\[%
\begin{array}
[c]{cl}
& 0\geq-2t_{0}x+\frac{\left(  t_{0}x\right)  ^{2}}{m\delta^{2}}-C_{4}t_{0}%
^{2}\\
\Longrightarrow & t_{0}x\leq2m\delta^{2}+C_{5}t_{0}\\
\Longrightarrow & G\leq2t_{0}x\leq4m\delta^{2}+2C_{5}t_{0}\\
\Longrightarrow & \left(
{\displaystyle\sum\limits_{j\in I_{d}}}
\left\vert e_{j}f\right\vert ^{2}+%
{\displaystyle\sum\limits_{\alpha\in\Lambda}}
\left(  1+\left\vert Y_{\alpha}f\right\vert ^{2}\right)  ^{\lambda}-\delta
f_{t}\right)  \left(  x_{0},t_{0}\right)  \leq\frac{4d\delta^{2}}{t_{0}}%
+C_{6}.
\end{array}
\]

$(ii)\ $If $x\leq%
{\displaystyle\sum\limits_{\alpha}}
\left(  1+\left\vert Y_{\alpha}f\right\vert ^{2}\right)  ^{\lambda}$, then
\[
0\geq-2C_{7}y^{2\lambda}+\frac{3}{4}\rho_{2}t_{0}y^{2}-C_{4}t_{0}.
\]

$\left(  a\right)  $ If $t_{0}<1$, then \
\[
y^{2}\left(  \frac{3}{4}\rho_{2}t_{0}-2C_{7}y^{2\left(  \lambda-1\right)
}\right)  \leq C_{4}t_{0},
\]
and
\[
\frac{3}{4}\rho_{2}t_{0}\leq\left(  C_{4}+2C_{7}\right)  y^{2\left(
\lambda-1\right)  },
\]
and%
\[
y\leq C_{2}t_{0}^{\frac{1}{2\left(  \lambda-1\right)  }},
\]
and%
\[
t_{0}y^{2\lambda}\leq C_{2}t_{0}^{\frac{2\lambda-1}{\lambda-1}}.
\]

$\left(  b\right)  $ If $t_{0}\geq1$, then%
\[
0\geq-2C_{7}y^{2\lambda}+\frac{3}{4}\rho_{2}t_{0}y^{2}-C_{4}t_{0}%
\]
and
\[
0\geq-2C_{7}t_{0}y^{2\lambda}+\frac{3}{4}\rho_{2}t_{0}y^{2}-C_{4}t_{0}a
\]
and
\[
0\geq-2C_{7}y^{2\lambda}+\frac{3}{4}\rho_{2}y^{2}-C_{4}%
\]
and%
\[
y\leq C_{8}%
\]
and
\[
t_{0}y^{2\lambda}\leq C_{9}t_{0}.
\]

Combining $\left(  a\right)  $ and $\left(  b\right)  $, we have
\[%
\begin{array}
[c]{cl}
& G\leq2t_{0}%
{\displaystyle\sum\limits_{\alpha}}
\left(  1+\left\vert Y_{\alpha}f\right\vert ^{2}\right)  ^{\lambda}\leq
C_{2}^{\prime}t_{0}+C_{3}^{\prime}t_{0}^{\frac{2\lambda-1}{\lambda-1}}\\
\Longrightarrow & \left(
{\displaystyle\sum\limits_{j\in I_{d}}}
\left\vert e_{j}f\right\vert ^{2}+%
{\displaystyle\sum\limits_{\alpha\in\Lambda}}
\left(  1+\left\vert Y_{\alpha}f\right\vert ^{2}\right)  ^{\lambda}-\delta
f_{t}\right)  \left(  x_{0},t_{0}\right)  \leq C_{2}+C_{3}t_{0}^{\frac
{\lambda}{\lambda-1}}.
\end{array}
\]

$(II)\ Case\ II:$ $x\leq0:$

We may assume%
\begin{equation}
\left(  \delta_{0}-1\right)
{\displaystyle\sum\limits_{j\in I_{d}}}
\left\vert e_{j}f\right\vert ^{2}\leq%
{\displaystyle\sum\limits_{\alpha\in\Lambda}}
\left(  1+\left\vert Y_{\alpha}f\right\vert ^{2}\right)  ^{\lambda};\label{62}%
\end{equation}
otherwise,
\[
F\left(  x_{0},t_{0}\right)  \leq0.
\]

By $\left(  1\right)  -\left(  4\right)  $, $\left(  \ref{59}\right)  $
becomes%
\begin{equation}%
\begin{array}
[c]{ccl}%
0 & \geq & -\frac{G}{t_{0}}+\frac{15}{16}\rho_{2}t_{0}\left(
{\displaystyle\sum\limits_{\alpha}}
\left\vert Y_{\alpha}f\right\vert ^{2}\right)  -\frac{2d^{2}}{\lambda\left(
2\lambda-1\right)  }\left(
{\displaystyle\sum\limits_{j}}
\left\vert e_{j}f\right\vert ^{2}\right)  \left(
{\displaystyle\sum\limits_{\alpha}}
\left(  1+\left\vert Y_{\alpha}f\right\vert ^{2}\right)  ^{\frac{1-\lambda}%
{2}}\right)  ^{2}\\
&  & -\frac{4d^{2}\left(  1+d\right)  }{\rho_{2}}t_{0}\left(
{\displaystyle\sum\limits_{j}}
\left\vert e_{j}f\right\vert \right)  ^{2}-d^{2}\lambda^{2}\left(
3-2\lambda\right)  ^{2}a^{2}t_{0}\left(
{\displaystyle\sum\limits_{\beta}}
\left(  1+\left\vert Y_{\beta}f\right\vert ^{2}\right)  ^{\lambda-1}\left\vert
Y_{\beta}f\right\vert \right)  ^{2}\\
&  & -\frac{\gamma^{2}d^{3}\left(  1+d\right)  }{2\lambda\left(
2\lambda-1\right)  }t_{0}\left(
{\displaystyle\sum\limits_{\beta}}
\left(  1+\left\vert Y_{\beta}f\right\vert ^{2}\right)  ^{\lambda-1}\left\vert
Y_{\beta}f\right\vert \right)  ^{2}\left(
{\displaystyle\sum\limits_{\alpha}}
\left(  1+\left\vert Y_{\alpha}f\right\vert ^{2}\right)  ^{\frac{1-\lambda}%
{2}}\right)  ^{2}\\
&  & -d\left(  a+a^{\prime}+acd^{4}\right)  \lambda\left(  3-2\lambda\right)
t_{0}%
{\displaystyle\sum\limits_{j,\beta}}
\left(  1+\left\vert Y_{\beta}f\right\vert ^{2}\right)  ^{\lambda-1}\left\vert
Y_{\beta}f\right\vert \left\vert e_{j}f\right\vert \\
&  & -\left(  db+dc+db^{\prime}+bcd^{5}+dc^{\prime}+b+c^{2}d^{4}\right)
\lambda\left(  3-2\lambda\right)  t_{0}%
{\displaystyle\sum\limits_{\alpha,\beta}}
\left(  1+\left\vert Y_{\alpha}f\right\vert ^{2}\right)  ^{\lambda
-1}\left\vert Y_{\alpha}f\right\vert \left\vert Y_{\beta}f\right\vert \\
&  & -4a\lambda t_{0}\left(
{\displaystyle\sum\limits_{\beta}}
\left(  1+\left\vert Y_{\beta}f\right\vert ^{2}\right)  ^{\lambda-1}\left\vert
Y_{\beta}f\right\vert \right)  \left(
{\displaystyle\sum\limits_{j}}
\left\vert e_{j}f\right\vert \right)  ^{2}\\
&  & -4\left(  1+b+c\right)  \lambda t_{0}%
{\displaystyle\sum\limits_{j,\alpha,\beta}}
\left(  1+\left\vert Y_{\beta}f\right\vert ^{2}\right)  ^{\lambda-1}\left\vert
e_{j}f\right\vert \left\vert Y_{\alpha}f\right\vert \left\vert Y_{\beta
}f\right\vert \\
&  & -\frac{t_{0}\kappa}{\lambda\left(  2\lambda-1\right)  }%
{\displaystyle\sum\limits_{\alpha}}
\left(  1+\left\vert Y_{\alpha}f\right\vert ^{2}\right)  ^{1-\lambda}\cdot%
{\displaystyle\sum\limits_{j}}
\left\vert e_{j}f\right\vert ^{2}-t_{0}\rho_{1}%
{\displaystyle\sum\limits_{j}}
\left\vert e_{j}f\right\vert ^{2}\\
& \geq & -\frac{G}{t_{0}}+\frac{3}{4}\rho_{2}t_{0}y^{2}+t_{0}[\frac{3}{16}%
\rho_{2}y^{2}-\frac{C_{9}}{\delta_{0}-1}y^{2}-C_{10}y^{2\lambda}%
-C_{11}y^{2\left(  2\lambda-1\right)  }-C_{12}y^{2\lambda}\\
&  & -C_{13}y^{3\lambda-1}-C_{14}y^{2\lambda}-C_{15}y^{4\lambda-1}%
-C_{16}y^{3\lambda}-\frac{C_{17}}{\delta_{0}-1}y^{2}-C_{18}y^{2\lambda}].
\end{array}
\label{2016aaa}%
\end{equation}
If we choose
\[
\delta_{0}=1+\frac{8}{\rho_{2}}\left(  C_{9}+C_{17}\right)
>1\text{\ \ \ \textrm{and}\ \ \ }\frac{1}{2}<\lambda<\frac{2}{3},
\]
then we derive the inequality%
\begin{equation}
0\geq-\frac{G}{t_{0}}+\frac{3}{4}\rho_{2}t_{0}y^{2}-C_{19}t_{0}\label{63}%
\end{equation}

for some constant $C_{19}>0$. Utilizing the same deductions as precedes and
$\left(  \ref{63}\right)  $ instead of $\left(  \ref{64}\right)  $, the proof
of this theorem is completed.
\end{proof}

\end{document}